\documentclass[12pt]{amsart}
\usepackage{amstext,amssymb}
\usepackage[all]{xy}
\usepackage{times}

\newtheorem{defn0}{Definition}[section]
\newtheorem{prop0}[defn0]{Proposition}
\newtheorem{thm0}[defn0]{Theorem}
\newtheorem{lemma0}[defn0]{Lemma}
\newtheorem{corollary0}[defn0]{Corollary}
\newtheorem{example0}[defn0]{Example}
\newtheorem{remark0}[defn0]{Remark}
\newtheorem{conjecture0}[defn0]{Conjecture}

\newenvironment{definition}{\begin{defn0}}{\end{defn0}}
\newenvironment{proposition}{\begin{prop0}}{\end{prop0}}
\newenvironment{theorem}{\begin{thm0}}{\end{thm0}}
\newenvironment{lemma}{\begin{lemma0}}{\end{lemma0}}
\newenvironment{corollary}{\begin{corollary0}}{\end{corollary0}}
\newenvironment{example}{\begin{example0}\rm}{\end{example0}}
\newenvironment{remark}{\begin{remark0}\rm}{\end{remark0}}

\newcommand{\propref}[1]{Proposition~\ref{#1}}
\newcommand{\thmref}[1]{Theorem~\ref{#1}}
\newcommand{\lemref}[1]{Lemma~\ref{#1}}

\newcommand{\exref}[1]{Example~\ref{#1}}
\newcommand{\secref}[1]{Section~\ref{#1}}
\newcommand{\remref}[1]{Section~\ref{#1}}

\def\EAN{ {({\bf k}^{\mathrm N},0)} }            
\def\MS{ {{\bf H}_{{\mathrm N,p}}} }             
\def\MF{ {{\underline{\bf H}}_{{\mathrm N,p}}} } 
\def\kmod{ {{\bf k}-{\bf mod}} }        
\def\ktopmod{ {{\bf k}-{\bf top.mod}} } 
\def\plim{ \underleftarrow{\lim} } 
\def\ilim{ {\lim_{ \stackrel{\longrightarrow}{n}}\mbox{ } } } 
\def\max{{\bf m}}                   
\def\maxn{{\bf n}}                  
\def\res{{\bf k}}                   

\begin{document}
\title[Moduli space]
{{A MODULI SCHEME OF EMBEDDED CURVE SINGULARITIES}}
\author[Juan Elias]{Juan Elias ${}^{*}$}
\thanks{
${}^{*}$Partially supported by MTM2007-67493.\\
\em \indent 2000 Mathematics Subject Classification, \rm 14H20,
14B07,  14D22}
\address{Departament d'Algebra i Geometria\newline
\indent Facultat de Matem\`{a}tiques \newline \indent Universitat
de Barcelona\newline \indent Gran Via 585, 08007 Barcelona, Spain}
\email{{\tt elias@ub.edu}}
 \date{\today}
\maketitle

\medskip
\begin{centerline}{
\begin{minipage}{13cm}
\tableofcontents
\end{minipage}}
\end{centerline}

\baselineskip 15pt

\medskip
\section*{Introduction}

A central problem in  Algebraic Geometry is  the classification of
several  isomorphism classes of objects by  considering their
deformations and studying the naturally related moduli problems,
see \cite{Mum80}, \cite{New78}. This general strategy has also
been applied to singularities. Some classes of singularities with
fixed numerical invariants are studied from the moduli point of
view, i.e. proving the existence of moduli spaces or giving
obstructions to their existence. See for instance \cite{GF94},
\cite{Lau79}, \cite{LP88} and \cite{Zar73}.

 The main purpose of this paper is to prove  the existence of  the moduli space $\MS$
parameterizing the embedded curve singularities of $\EAN$ with an
admissible Hilbert polynomial $p$  and to  study its
 basic properties. The main difference between the
classical projective moduli problems and the case studied here is
that ${\bf H}_{N,p}$  is not  a locally finite type scheme.
Hence the general techniques of construction of moduli spaces of
projective objects do not apply to our problem and we need to
develop specific ones. Since $\MS$ is a projective limit of
$\res$-schemes of finite type we  define  a measure $\mu_p$ in $\MS$
valued in the  completion $\widehat{ \mathcal M}$ of the ring
${\mathcal M}=K_0({\bf Sch})[\, \mathbb L^{-1}]$ where $\mathbb L$
is  the class of $K_0({\bf Sch})$ defined by
the affine line over $\res$.
This measure induces a motivic integration on $\MS$ and enable us to consider a motivic
volume for singularities of arbitrary dimension.
 See \cite{Kon95}, \cite{DL3ECM}, and \cite{Loo02} for the motivic integration on jet schemes.

 In \cite{Eli90}, see also \cite{Eli01}, we characterized  the Hilbert-Samuel
polynomials of curve singularities: we proved that there exists a
curve singularity $C$ with  embedding dimension $b$ and Hilbert
polynomial $p=e_0T-e_1$  if, and only if, either $b=e_0=1$ and
$e_1=0$, or
 $2 \le b \le e_0$, and
 $\rho_{0,b,e_0} \le e_1 \le \rho_{1,b,e_0},$
 see \thmref{compositio} for the definitions of $\rho_{0,b,e_0}$ and $\rho_{1,b,e_0}$.
  Moreover,
for each triplet $(b,e_0,e_1)$ satisfying the above conditions
there is a reduced curve singularity $C \subset {({\bf k}^N,0)}$
with ${\bf k}$ an  algebraically closed field. From this result
and the main result of this paper,
 \thmref{monster},
we deduce that the moduli space ${\bf H}_{N,p}$ is nonempty for
the polynomials  $p=e_0T-e_1$ with  $\rho_{0,b,e_0} \le e_1 \le
\rho_{1,b,e_0}$ for some $b \le N$.


\medskip
\noindent The contents of the paper is the following. The main
purpose of \secref{chartrunc} is to characterize the zero
dimensional closed subschemes $Z \subset \EAN$ for which there
exists a curve singularity $C \subset \EAN$ such that $Z$ is a
truncation of $C$.
In other words we
characterize which zero-dimensional schemes can be lifted to a
curve singularity, \thmref{big-lifting}. The key idea in the proof
of \thmref{big-lifting} is  the control of the dimension of some
lifting of $Z$ obtained  by applying Artin's approximation theorem to
the system of equations defined by some syzygy conditions deduced
from Robbiano-Valla's characterization of standard basis.

It is well known that some properties $\mathcal P$ defined in the
set of curve singularities are finitely determined, i.e. determined
by the $n$-th truncation $C_n$ of $C$ for  $n \ge n_0=n_0(\mathcal
P)$. The most studied finitely determined property is the analytic
type. In \cite{Eli86b} we prove that analytic type is finitely
determined for $n \ge  n_0=2 \mu +1$,  where $\mu$ stands for the
Milnor number of $C$, as a corollary we get that if a property of
curve singularities is invariant by analytic transformations then is
finitely determined. In \cite{Eli86b} we also  prove that "to have
the same tangent cone" or "to have the same Hilbert function" are
finitely determined properties. In the section 3 we attach to any
finitely determined property $\mathcal P$ a rational power series
$MPS_{\mathcal P} \in  {\mathcal M}[T]_{loc}$,
\propref{motivicpoincare}.


In the second section  we introduce the  algebraic families of
curve singularities over a scheme $S$. Notice that the concept of
family is  a key ingredient in a moduli problem.
 We analyze the relationship
between families  and normally flat morphisms and we also give
several explicit examples of families of curve singularities with
fixed Hilbert  polynomial.


The purpose of \secref{somesub} is to construct a moduli scheme
$\MS$   parameterizing the embedded curve singularities of $\EAN$
with fixed Hilbert  polynomial $p$. In the main result of this
section, \thmref{monster}, we establish  the existence of  a ${\bf
k}-$scheme $\MS$ pro-representing the functor of families $\MF$.
We will obtain $\MS$ as an inverse limit of ${\bf k}-$schemes
$\Xi_{n}$ of finite type with affine  morphisms $a_n:\Xi_{n} \to
\Xi_{n-1}$. Notice that in \propref{eles} we prove some properties
of $\Xi_{n}$ as a corollary of \thmref{big-lifting}; in particular
we prove that $\Xi_{n}$ contains the $n$-th truncations of all
curve singularities with Hilbert  polynomial $p$. The key point in
the existence of $\MS$ is the control of the behavior of the
degree one superficial elements given in \cite{Eli86b}; this
enables us to prove that $a_n$ is affine for a big enough $n$. As
a corollary we get that  the cohomological dimension of $\MS$ is
finite, and that there exists a universal family over $\MS$.
 We
end section three by constructing the Hilbert  strata of $\MS$ for
each admissible Hilbert  function, in particular we prove the
existence of a moduli space parameterizing normally flat families.


In the second part of section three we introduce a motivic
measure  $\mu _p$ defined in the algebra of cylinders $\MS$ valued in $\mathcal
 M$.
 By means of $\mu_p$ we  define for a singularity $X$ of arbitrary dimension a
 motivic volume $vol(X) \in \widehat{\mathcal M}$.
Given a  property $\mathcal P$ defined on the set of curve singularities
 with Hilbert polynomial $p$  we define a motivic
 Poincare  series  $MPS_{\mathcal P} \in {\mathcal M}[[T]]$.
 We prove that if $\mathcal P$  is a finitely determined property
 then  $MPS_{\mathcal P}\in  {\mathcal M}[T]_{loc}$.
 In particular we prove $MPS_{\MS}\in  {\mathcal M}[T]_{loc}$.


In the section 4 we compute the tangent space of $\MS$ at a closed
point, for this we determine  families that are first order
deformations. We apply these results to the moduli space of
singularities with maximal Hilbert  function, in particular to
plane curve singularities. By considering $\MS$ as object of the
category of pro-schemes we define  a topology  on its sheaf of rings. Taking dual spaces
with respect this topology we obtain the reflexivity of the
tangent space at the closed points of $\MS$ and also the
reflexivity of the normal space of the curve singularities
parameterized by $\MS$.
 We end the paper studying the
obstructiveness of the closed points of $\MS$. We prove that plane
curve singularities and space curve singularities with maximal
numbers of generators with respect to their multiplicity define
non-obstructed closed points.

\medskip
\noindent {\bf {\sc Acknowledgements:}} The author would like to
thank
O.A. Laudal for the comments and suggestions
that improved this paper.


\medskip
\section{Truncations of curve singularities.}
\label{chartrunc}

Throughout this paper ${\bf k}$  is an algebraically closed
field. We set $R={{\bf k}[[X_1,...,X_N]]}$, $M=(X_1,...,X_N)$ is
the maximal ideal of $R$, and we denote by $\EAN$ the ${\bf
k}-$scheme $Spec(R)$.

A curve singularity of $\EAN$ is a one-dimensional Cohen-Macaulay,
closed subscheme $C$ of $\EAN$. We denote by $\max$ the maximal
ideal of ${\mathcal O}_C=R/I$, and by $H^1_{C}$ (resp.
$h^1_{C}(T)=e_0(T+1)-e_1$) the first  Hilbert function (resp.
Hilbert polynomial) of $C$, i.e. $H^1_{C}(t):= length_R({\mathcal
O}_C/\max^{t+1})$ and $H^1_{C}(t)=h^1_{C}(t)$ for $t \ge e_0-1$;
$e_0$ is the multiplicity of $C$. An element $x\in {\mathcal O}_C$
is a degree one superficial element if $ (\max^{n+1}:x)=\max^{n} $
for all $n \gg 0$, see for instance \cite{Nor57}.

From now on we fix a degree-one polynomial $p(T)=e_0(T+1)-e_1$ for
which
 there exist a curve singularity $C \subset \EAN $
of embedding dimension $b \le N$ with $h^1_{C}=p$, see
\propref{compositio}.

 Given a curve singularity $C$ we
denote by $C_{n}$ the closed sub-scheme of $\EAN$
defined by the ideal $I(C)+M^{n}$, we  say that $C_{n}$ is the
$n-$th truncation of $C$, $n \ge 1$.
First we recall some necessary conditions for the
ideals $J \subset R$ to being a truncation of a curve singularity.

\medskip
\begin{lemma}\label{nec-cond}
  Let  $C \subset \EAN$ be a curve singularity of multiplicity
  $e_0$.
  There exists a linear form $L\in M\setminus M^2$ such that for all
  $n \ge e_0+1$ such that the following conditions hold:
\begin{enumerate}
  \item[(1)] $length_R(R/I(C)+M^n+(L)) =e_0$,

  \item[(2)] if $\maxn$ is  the maximal ideal of $R/I(C)+M^n$
    then for all $t$, $n-2 \ge t \ge e_0-1$, the product by $L$
        defines  an isomorphism of $ \res$-vector spaces of
        dimension $e_0$:
        $$
        \frac{\maxn^t}{\maxn^{t+1}} \stackrel{.L}{\longrightarrow}
        \frac{\maxn^{t+1}}{\maxn^{t+2}}
        $$
\end{enumerate}
\end{lemma}
\begin{proof}
Since $\res $ is infinite and ${\mathcal O}_C$ is a Cohen-Macaulay
local ring we may assume that there exists a linear form $L \in
M\setminus M^2$ such that its coset defines a degree one superficial
element of ${\mathcal O}_C$, \cite{Mat77} Proposition 3.2. Then we
have $length_R(R/I(C)+(L)) =e_0$ and  $M^n \subset I(C)+(L)$ for all
$n \ge e_0$. From this we deduce the first equality.

From \cite{Kir75}, Theorem 2, we have
$$
dim_{\res}\left(\frac{\max^n}{\max^{n+1}}\right)=e_0
$$
for all $n\ge e_0-1$. Hence from \cite{Eli86b}, Proposition 1, we
deduce
$$
\frac{\max^t}{\max^{t+1}} \stackrel{.L}{\longrightarrow}
        \frac{\max^{t+1}}{\max^{t+2}}
$$
is an isomorphism of $ \res$-vector spaces of  dimension $e_0$ for
all  $n-2 \ge t \ge e_0-1$. Since
$$
\frac{\max^t}{\max^{t+1}} \cong \frac{\maxn^{t}}{\maxn^{t+1}}
$$
for all  $n-2 \ge t \ge e_0-1$, we get $(2)$.
\end{proof}

\medskip
 Next we define a set of ideals ${\mathbb T}_n$ containing
the $n$-th truncation of curve singularities of multiplicity $e_0$.
Since we want to consider in \remref{somesub} a scheme structure on
some subsets of ${\mathbb T}_n$,   we replace the identity of
\lemref{nec-cond} $(1)$ by an inequality that will define an open
condition on a suitable Grassmanian.

\medskip
\begin{definition}
Let $n \ge e_0+1$ be an integer, ${\mathbb T}_n$ is the set of
ideals $J \subset R$ such that $M^n \subset J$ and such that there
exists a linear form $L \in M \setminus M^2$
  such that
        \begin{enumerate}
        \item[(1)]
                $length_R(R/J+(L)) \le e_0$, and
        \item[(2)]
        if $\maxn$ is the maximal ideal of $R/J$ then the
        product  by $L$ is an isomorphism of $\res$-vector spaces
        $$
        \frac{\maxn^t}{\maxn^{t+1}} \stackrel{.L}{\longrightarrow}
        \frac{\maxn^{t+1}}{\maxn^{t+2}}
        $$
        of dimension $e_0$, for all $t=e_0-1,\dots , n-2$.
        \end{enumerate}
\end{definition}

\medskip
\noindent
 From
the condition $(2)$ it is easy to prove that  there exist a linear
polynomial $q_J(T)= e_0( T+1) - b$, $b \in \mathbb Z$, such that
$$ q_J(t)= length_R(R/ J + M^{t+1}) $$
 \noindent for all $t=e_0-2,\dots , n-1$.
From the characterization of Hilbert functions due to Macaulay,
see for instance \cite{Sta78}, we get
$$
q_J(t) \ge e_0(t+1) -  \binom{e_0}{2}.
$$

\noindent For all $n_1 \le n_2$ we denote by $a_{n_2,n_1}:
{\mathbb T}_{n_2} \longrightarrow {\mathbb T}_{n_1} $ the
projection map $a_{n_2,n_1}(J)=J+M^{n_1}$.

\bigskip
\noindent
 For all $f \in R$ we denote by $f^*\in S=\res[X_1,\dots, X_N]$ the initial form of
$f$. If $J$ is an ideal of $R$ we will denote by $J^*$ the
homogeneous ideal of $S$ generated by the initial forms of the
elements of $J$, we put $Gr(R/J)=S/J^*$ for the associated graded
ring to $R/J$. A set of elements of $J$ such that their initial
forms is a (minimal) set of generators of $J^*$ is known as a
(minimal) standard basis of $J$. We will denote by $S_n$, resp.
$J^*_n$, the degree $n$ component  of $S$,  resp. $J^*$.

\medskip
\begin{proposition}
\label{shape-std} Let $J$ be an element of \quad ${\mathbb T}_n$, $n
\ge e_0+2$, then  every  minimal homogeneous basis $F_1,\dots, F_s$
of $J^*$ satisfies  $$deg (F_i) \not\in  \{e_0+1,\dots, n-1\}$$
\noindent for all $i=1,\dots, s$.
\end{proposition}
\begin{proof}
Let $\maxn$ be the maximal ideal of $R/J$.
 $L$ is a linear form,
satisfying the conditions $(1)$ and $(2)$ of the definition of
${\mathbb T}_n$, so  that $$
\frac{S_t}{J^*_t}=\frac{\maxn^t}{\maxn^{t+1}}
\stackrel{.L}{\longrightarrow}
\frac{S_{t+1}}{J^*_{t+1}}=\frac{\maxn^{t+1}}{\maxn^{t+2}} $$
\noindent is an isomorphism of $\res$-vector spaces for all
$t=e_0-1,\dots , n-2$. From the surjectivity  of these morphisms we
get $S_t=L S_{t-1} + J_t^*$ for $t=e_0,\cdots n-1$. Hence we deduce
$$
J_t^*\subset S_1  J_{t-1}^* + L S_{t-1}
$$
$t=e_0+1,\cdots n-1$.
Let $a$ be an element of $J_t^*$, then there exist
$b\in S_1  J_{t-1}^* \subset J_t^*$ and $\alpha \in S_{t-1}$ such that
$a=b + L \alpha.$ In particular $L \alpha = a - b \in J_{t}^*$.
From the injectivity of the above morphisms we get $\alpha \in J_{t-1}^*$,
so $a\in S_1 J_{t-1}^*$.
Hence $J_t^*\subset S_1 J_{t-1}^*$  and then
$$J_t^*= S_1 J_{t-1}^*$$
\noindent
 for $t=e_0+1,\cdots , n-1$.
From this we get the claim.
\end{proof}

\medskip
\begin{definition}
Let $J \in {\mathbb T}_n$ be an ideal, $n \ge e_0+2$, and let
$F_1,\dots, F_s$ be a minimal homogeneous basis of $J^*$. We may
assume that $deg(F_i) \le e_0$ for $i=1,\dots , v$ and $deg(F_i)
\ge n$ for $i=v+1,\dots , s$. We denote by $\tilde J$ the
homogeneous ideal of $S$ generated by $F_i$, $i=1,\dots, v$.
\end{definition}


\medskip
 Next we will recall a result of G. Hermann, \cite{Her26},
quoted by M. Artin in \cite{Art69b}, Theorem 6.5. We need some
additional definitions. The degree $deg({\mathbf f})$ of a $r$-pla
of polynomials ${\mathbf f}=(f_1,\dots , f_r) \in S^r$ is by
definition the sum of the degrees of $f_1,\dots , f_r$. The degree
of ${\mathbf F}=\{{\mathbf f}_1,\dots , {\mathbf f}_s\}$,
${\mathbf f}_i \in S^r$, is the sum of the degrees of ${\mathbf
f}_1,\dots , {\mathbf f}_s$. Let $B \subset S^r$ be a
$S$-sub-module, the degree $deg(B)$ of  $B$ is the minimum of the
degrees of its systems of generators.

\medskip
\begin{proposition}[\cite{Her26}, \cite{Art69b}]
\label{hermann} There exists an integer valued function $\gamma :
{\mathbb N}^2 \longrightarrow {\mathbb N}$ such that for all ideals
$K \subset S$ of degree $\le d$ there exists a primary decomposition
of  $K= K_1 \cap \dots \cap K_r$ such that the following integers
are bounded by $\gamma (N,d)$:
\begin{enumerate}
  \item[(1)] The number $r$, and the degree of each primary ideal $K_i$.

  \item[(2)] The degree of the associated prime ideal $p_i =rad(K_i)$
  and the exponents $m_i$ such that $p_i^{m_i} \subset K_i$,
  $i=1,\cdots,r$.
\end{enumerate}
\end{proposition}

\medskip
We will apply the last proposition to our setting.

\begin{definition}
Let $K$ be a height $N-1$ homogeneous ideal of $S$. Given a minimal
primary decomposition $K=K_1 \cap \dots \cap K_r$ under the
conditions $(1)$ and $(2)$ of \propref{hermann}, we can split this
decomposition, after a suitable permutation, in two pieces
$K_{emb}=K_{w+1} \cap \dots \cap K_r$, such that $rad(K_{emb})=M$
and $K_{no-emb}=K_1 \cap \dots \cap K_w$ is a perfect height $N-1$
ideal of $S$.
\end{definition}

\medskip
\begin{proposition}
\label{jotatilde} There exists a function $\delta :{\mathbb N}^2
\longrightarrow {\mathbb N}$ such that: let $J$ be an ideal of
$\;{\mathbb T}_n$, $n \ge \delta (N,e_0)$, then the following
conditions hold:
\begin{enumerate}
  \item[(1)] $S/\tilde J$ is a one-dimensional graded ring of
  multiplicity $e_0$, $deg(\tilde J) \le \delta(N,e_0)$,
                and ${\tilde J} + M^n = J^*$.

  \item[(2)] If $J$ is the ideal defining the
          $n$-th truncation of a curve singularity $C$ of multiplicity $e_0$
          then $S/\tilde J$ is the associated graded ring to ${\mathcal O}_C$,
          i.e. $\tilde J = I(C)^*$.
  \item[(3)]   If $Syz_1(\tilde J)$ is the first syzygy
  module of $F_1,\dots , F_v$ then it holds
  $deg(Syz_1(\tilde J)) \le  \delta(N,e_0).$
\end{enumerate}
\end{proposition}
\begin{proof}
\noindent $(1)$ from the definition of $\tilde J$ we get ${\tilde
J} + M^n = J^*$. Let us assume that $S/\tilde J$ is a zero
dimensional ring. Since $\tilde J$ is generated by homogeneous
forms of degree $ \le e_0$ we have
$$ deg(\tilde J) \le e_0
\binom{N+e_0}{N}=d ,$$
 \noindent so from \propref{hermann} $(2)$
we get that
$$ length_S(S/\tilde J) \le
\binom{N-1+\gamma(N,d)}{N-1}=\eta(N,e_0). $$

We define $\delta(N,e_0)=\eta(N,e_0)+\binom{e_0}{2} +1$.
 Since $length_S(S/\tilde J) \ge e_0(n) - \binom{e_0}{2} $ then for
 $n \ge   \delta(N,e_0)$ we get a
contradiction, so $dim(S/\tilde J)\ge 1$.

 Let
$\maxn$ be the maximal ideal of $S/\tilde J$. From the definition of
$\tilde J$ and condition $(2)$ of the definition of ${\mathbb T}_n$
we get that $\maxn^{t+1}=L \maxn^t$ for all $t\ge e_0-1$, so
$dim(S/\tilde J) \le 1$. Since $dim(S/\tilde J)\ge 1$ we have that
$S/\tilde J$ is a one-dimensional graded ring of multiplicity less
or equal than $e_0$.
 Let
${\tilde J}= K_{no-emb} \cap K_{emb}$ be a primary decomposition
of $\tilde J$ satisfying  \propref{hermann}, where $K_{no-emb}$ is a perfect height $N-1$ ideal
and $K_{emb}$ is a $M-$primary ideal.
Since $\delta(N,e_0) \ge \gamma(N,d)$ from \propref{hermann} we
have $M^n \subset K_{emb}$ and then
$$ H^1_{S/\tilde J}(n)=
H^1_{S/K_{no-emb}}(n) + length_S\left(\frac{K_{no-emb}+K_{emb}}{K_{emb}}\right).
$$ \noindent From this we deduce that $H^0_{S/\tilde J}(n)=
H^0_{S/K_{no-emb}}(n)$ for all $n \ge \delta(N,e_0)$.
 Since $S/K_{no-emb}$ is a one-dimensional
Cohen-Macaulay graded ring of multiplicity
$e_0(S/K_{no-emb})=e_0(S/\tilde J) \le e_0$ and $H^0_{S/\tilde
J}(e_0-1)=e_0$, second condition of the definition of ${\mathbb
T}_n$, we get that $S/\tilde J$ has multiplicity $e_0$.

\noindent
 $(2)$ see \cite{Eli86b}, Proposition 2;  $(3)$  follows from \cite{Sjo92}.
\end{proof}

\medskip
We denote by $\beta_{Artin} : \mathbb N ^4 \longrightarrow \mathbb
N$ the so-called beta function of Artin, see \cite{Art69b},
Theorem 6.1.

\medskip
\begin{proposition}
\label{betabeta} There exists a numerical function $\beta: \mathbb
N^3 \longrightarrow \mathbb N$ such that for all ideals $J$ of $\;
{\mathbb T}_n$, $n \ge \delta (N,e_0)$,
$$\beta(N,e_0,n) \ge
\beta_{Artin}(N, v(r+1), 2 r, n)$$

\noindent
with $v$ the minimal number of generators of $\tilde J$, and $r$ the minimal number of
generators of $Syz_1(\tilde J)$.
\end{proposition}
\begin{proof}
From the last result we know that $v\le deg(\tilde J)\le \delta(N, e_0)$
and $r\le deg(Syz_1(\tilde J)) \le  \delta(N,e_0)$.
From these inequalities we deduce the claim.
\end{proof}

\medskip
\begin{theorem}
\label{big-lifting} For all $n \ge \delta(N,e_0)$ the set of the
associated ideals to $n$-th truncations of curve singularities of
multiplicity $e_0$ coincides with
 $${\mathbb T}^{'}_n=a_{\beta(N,e_0,n),n}({\mathbb T}_{\beta(N,e_0,n)}).$$
\noindent Moreover, given $J \in {\mathbb T}^{'}_n$ there exists a
curve singularity $C$ of multiplicity $e_0$ such that
${\mathcal O}_{C_n}=R/J$,
and $S/\tilde J$ is the associated graded ring to ${\mathcal O}_C$.
\end{theorem}
\begin{proof}
Let us consider an ideal $J=(f_1,\dots, f_s)$ of ${\mathbb T}_{\overline n}$,
$\overline n =\beta(N,e_0,n)$, such that $F_i=f_i^{*}$, $i=1,\cdots, s$,
form a minimal basis of $J^*$. We may assume that $deg(F_i) \le e_0$ for $i=1,\dots , v$ and $deg(F_i)
\ge \overline n$ for $i=v+1,\dots , s$, \propref{fibration}.
 We denote by $\tilde J$ the
homogeneous ideal of $S$ generated by $F_i$, $i=1,\dots, v$.

Let $Z$ be the first syzygy module of
$f_1, \dots, f_s$, and let $Z^*$ be the first syzygy module of
$F_1=f_1^*,\dots , F_s=f_s^*$. Let us recall that there exists a
map $\Phi: Z \longrightarrow Z^*$, see proof of \cite{RV80}
Theorem 1.9, such that for all $(a_1,\dots,a_s) \in Z$ we have
$\Phi(a_1,\dots,a_s)=(b_1,\dots,b_s)$ with $b_i$ the initial  form
of $a_i$ if $deg(a_i)=p-d_i$, with $d_i=deg(F_i)$,
$p=Min\{deg(a_i)+d_1, i=1,\dots,s\}$, and zero otherwise.

 Let  $R^1,\dots ,   R^r$ be
a minimal system of generators of syzygy module of $F_1,\dots , F_v$.
 From
\cite{RV80}, Theorem 1.9, there exist elements $\hat R^1,\dots ,
\hat R^r$ of $Z$ such that $\Phi(\hat R^i)= R^i$, $i=1,\dots,r.$
Let $\tilde R^i$ the projection on the first $v$ components of
$\hat R^i$, $i=1,\dots,r$. Then we have, $i=1,\dots,r$,
 $$ \sum_{j=1}^{v} \tilde
R^i_j f_j =0  \quad \text{mod }(X_1,\dots,X_N)^{\overline n}.$$

Let us consider the following system of equations attached to these
syzygy conditions $$(Syz): \left\{
\begin{array}{ll}
     \sum_{j=1}^{v} X^i_j Y_j =0     &  \\
     i=1,\dots,r                            &
  \end{array}
\right. $$ \noindent considered in the polynomial ring
$\res[X_1,\dots,X_N; Y_1,\dots,Y_v,X_1^1,\dots,X^r_v]$.

Since $\overline n =\beta(N,e_0,n)\ge \beta_{Artin}(N, v(r+1), 2 r, n)$,
from the approximation theorem of Artin, \cite{Art69b} Theorem 6.1, there
exists a solution of the system of equations $(Syz)$ $$
\sum_{j=1}^{v} {\overline R}^i_j g_j =0$$ \noindent $i=1,\dots,r$,
and such that $$ \left\{
\begin{array}{ll}
     g_j=f_j  \text{ mod } M^n & i=1,\dots,v\\
     {\overline R}^i= \tilde R^i   \text{ mod } M^n& i=1,\dots,r
  \end{array}
\right. $$

\noindent Let us define $I=(g_1,\dots,g_v)R$. Next step is to
prove that $g_1,\dots,g_v$ is a standard basis of $I$. We will
prove it by means of \cite{RV80} Theorem 1.9.
 Notice that $I=J\text{ mod } (X_1,\dots,X_N)^n$, $g_i^*=f_i^*=F_i$, $i=1,\dots,v$.
 Hence $R^1,\dots ,R^r$
is also a minimal system of generators of the first syzygy module of
$\{g_1^*,\dots,g_v^*\}$.
Since $\Phi({\overline R}^i)=\Phi(\tilde R^i)=R^i$,
$i=1,\dots,r$, and $\overline R^1, \dots, \overline R^r$ verifies
$(Syz)$,  from \cite{RV80} Theorem 1.9, we get that
$\{g_1,\dots,g_v\}$ is a standard basis of $I$, i.e.
$I^*=(g_1^*,\dots,g_v^*)=(F_1,\dots,F_v)=\tilde J$.
In particular $R/I$ is a one-dimensional local ring of multiplicity $e_0$,
\propref{jotatilde}  $(1)$.

From the condition $(1)$ of the
definition of ${\mathbb T}_{\overline n}$ we get that there exists a linear
form $L$ such that $dim_{\res}(R/I+(L)) = e_0$, so $R/I$
is a one-dimensional Cohen-Macaulay local ring of multiplicity $e_0$. If we define
$C=Spec(R/I)$ then we deduce the claim.
\end{proof}

\medskip
\section{Families of embedded curve singularities.}

For all ${\bf k}$-scheme  of finite type $S=Spec(A)$ we will
denote by ${({\bf k}^N,0)}_{S}$ the affine ${\bf k}$-scheme
$Spec(A[[X_1,...,X_N]])$. We denote  by $\pi : {({\bf k}^N,0)}_{S}
\longrightarrow S$ the morphism of ${\bf k}$-schemes induced by
the natural morphism of ${\bf k}$-algebras $A \longrightarrow
A[[\underline X]]=A[[X_1,...,X_N]].$ Given a closed subscheme $Z
\subset {({\bf k}^N,0)}_{S}$ we will denote by  $Z _{n}$ the
closed sub-scheme of ${({\bf k}^N,0)}_{S}$ defined by the ideal
$I(Z_{n})=I(Z)+(\underline X)^{n}$, for all $n \ge 1$.

\medskip
Let us denote by $\MF:{\bf{Aff}} \longrightarrow {\bf{Set}}$ the
contravariant functor such that for all affine ${\bf k}$-scheme of
finite type $S$

$$ \MF (S) = \left\{
\parbox{11.2cm}{
\vskip1mm \quad closed subschemes $Z \subset {({\bf k}^N,0)}_{S}$
such that $\pi : Z \longrightarrow S$

\noindent \quad is flat and

\vskip2mm
\begin{enumerate}
\item[(i)]
 for all $n \ge e_0+1$  the morphism $\pi_n :  Z_{n}
\longrightarrow S$ is flat with fibers of length $p(n-1)$,
\item[(ii)]
for all closed points $s \in S$ the fiber $Z_s= Z\otimes
_{S}\res(s)$ is a curve singularity of ${({\bf k}^N,0)}$ with
Hilbert polynomial $p$.
\end{enumerate}
\vskip1mm } \right\} $$

\medspace \noindent we say that  $\underline{\bf{H}}_{N,p}(S)$ is
the set of families of curves over $S$ with Hilbert polynomial $p$.

\medskip
\begin{proposition}
\label{critfam}
\begin{enumerate}
\item[(1)] $\MF (Spec({\bf k}))$ is the set of curve singularities of
           ${({\bf k}^N,0)}$ with Hilbert polynomial $p$.
\item[(2)]
Given a scheme $Z \subset {({\bf k}^N,0)}_{S}$ the following
conditions hold

\noindent (2.1) if $Z$  verifies the condition $(i)$ of the
definition of family of curve singularities over $S$ then
 $Z$  is  flat over $S$,

\noindent (2.2) if $S$ is reduced and if for all closed points $s
\in S$ the fiber $Z_{s}$ is a curve singularity with Hilbert
polynomial $p$ then $Z$  is a family of curves over $S$.
\end{enumerate}
\end{proposition}
\begin{proof}
(1) From the definition of family of curve singularities it is
easy to see that  $\MF (Spec({\bf k}))$ is a set of curve
singularities with Hilbert polynomial $p(T)$. Let $C$ be a curve
singularity of ${({\bf k}^N,0)}$ with Hilbert polynomial $p(T)$.
From \cite{Kir75}, Theorem 2, we have $h^1_{C}(n)=H^1_{C}(n)$ for
$n \ge e_0-1$, so we have $C \in \MF (Spec({\bf k}))$.

\noindent (2.1) follows the main ideas of the proof of
\cite{MatCA}, Theorem 55. Let $Z=Spec(A[[\underline X]]/J)$ be a
closed subscheme of ${({\bf k}^N,0)}_{S}$, with $S=Spec(A)$, such
that
 for all $n \ge e_0+1$  the morphism $\pi_n:  Z_{n}
\longrightarrow S$ is flat with fibers of length $p(n-1)$. We have
to prove that the morphism
$$
A \longrightarrow B:= \frac{A[[\underline X]]}{J}
$$
is flat. Let $f: L \longrightarrow P$ be a monomorphism of finitely
generated $A$-modules, we have to prove that $f\otimes
Id_B:L\otimes_A B \longrightarrow P\otimes_A B$ is also a
monomorphism.

Since the morphism $A \longrightarrow B_n:= \frac{A[[\underline
X]]}{J+(\underline X)^n}$ is flat for all $n\ge e+1$, we have that
$$
f\otimes Id_{ B_n}:L\otimes_A  B_n \longrightarrow P\otimes_A B_n$$
 is also a monomorphism, $n\ge e_0 +1$.
 Hence
$$
\plim f\otimes Id_{ B_n}: \plim \left( L\otimes_A B_n\right)
\longrightarrow \plim \left( P\otimes_A B_n\right)$$
 is  a monomorphism.
Since the $A$-modules $L$ and $P$ are finitely generated
it is well known that
$$\plim \left( L\otimes_A  B_n\right) \cong L\otimes_A B, \;\;\;
\plim \left( P\otimes_A  B_n \right)\cong P\otimes_A B$$ from this
we deduce $(2.1)$.  The statement (2.2) follows from \cite{Kir75},
Theorem 2.
\end{proof}

\medspace
\begin{remark}
Notice  that the condition $(i)$ implies that for all closed
points $s \in S$ the fiber $Z_{s}$ is a $1-$dimensional closed
sub-scheme of ${({\bf k}^N,0)}$ with Hilbert polynomial $p$. Hence
$(ii)$ can be changed to

\medskip
\noindent (ii)' for all closed points $s \in S$ the fiber $Z_{s}$
is a Cohen-Macaulay scheme.
\end{remark}

\medspace
\begin{remark}
There exist flat morphisms that are not families. See
\exref{flat-nfam} for a first order deformation of a plane curve
singularity  that is not a family.
\end{remark}

\medskip
\noindent In the following examples we will construct families of
curve singularities with a closed  fiber $C$ by deforming the ideal
$I(C)$, by deforming a first syzygy matrix of $I(C)$ and by
deforming a parametrization of $C$.

\medspace
\begin{example}
Let $F, G_{1},..., G_{r}$ be power series in the variables $X_{1},
X_{2}$. We assume that $order(F)=e_0$ and $order(G_{i})\ge e_0+1$
for $i=1,...,r$. We put $A={\bf k}[T_{1},...,T_{r}]$, $S=Spec(A)$,
$I=(F+ \sum_{i=1}^{r}T_{i}G_{i}) \subset A[[X_{1},X_{2}]],$ and
$Z=Spec(A[[X_{1},X_{2}]]/I)$.
 Notice that for all point $s$ of
$S$ the fiber $Z_{s}$ is a plane singularity of multiplicity $e_0$.
From \propref{critfam} (2.2) we deduce that  $Z$ is a family of
plane curve singularities with Hilbert polynomial $p=e_0T-
e_0(e_0-1)/2$.
\end{example}

\medspace
\begin{example}
Let us consider the curve singularity of ${({\bf k}^3,0)}$ defined
by the ideal generated by the maximal minors of the matrix
(\cite{EN62}) $$ \left(
\begin{array}{ll}
X_{3}       & 0     \\
 X_{1}^{e_0-1} & X_{3} \\
  0 & X_{2} \\
\end{array}
\right) $$ \noindent An straightforward  computation shows that
$h^1_{C}=e_0T-(e_0^{2}-3e_0+4)/2$, and
$$H^1_{C}=\{1,3,4,5,6,...,e_0-1,e_0,e_0,...\}.$$

\noindent Let us consider the closed subscheme $Z$ of ${({\bf
k}^N,0)}_{S}$, $S=Spec({\bf k}[U])$, defined by the ideal
generated by the maximal minors of the matrix $$ \left(
\begin{array}{ll}
X_{3}+ Q_{1} U      & 0              \\
X_{1}^{e_0-1}+Q_{4} U & X_{3}+ Q_{2} U \\
0                   & X_{2}+ Q_{3} U \\
\end{array}
\right) $$ \noindent where $Q_{1}, Q_{2}, Q_{3}$ are formal power
series in $X_{1}, X_{2}, X_{3}$ of order at least $2$, and $Q_{4}$
is a power series in the same set of variables  of order at least
$e_0$. We know that $Z$ is a flat deformation of $C$ with base $S$,
see \cite{Art76}.
 We have a better result: from \cite{RV80},
Theorem 1.9, the tangent cone of the fiber $Z_{s}$ coincides with
$Z_{0}=C$ for all closed point $s \in S$. Hence the Hilbert
function is constant on the fibers, from \propref{critfam} we get
that $Z$ is a family of curve singularities with Hilbert
polynomial $p=e_0T-(e_0^{2}-3e_0+4)/2$.
\end{example}

\medspace
\begin{example}
Notice that in the previous example we have deformed the matrix of
syzygies of the curve singularity obtaining a family of curve
singularities. We can also deform the parametrization, i.e. the
normalization morphism, in order to obtain families. In this case
the families verify a stronger condition the singularity order of
the fibers is constant, see \cite{Tei77}. Let $C$ the curve
singularity of ${({\bf k}^4,0)}$ with normalization morphism $
{\mathcal O}_C  \longrightarrow {\overline {\mathcal O}_C} \cong
{\bf k}[[t]], $ \noindent defined by
$$X_{1}=t^{6},
X_{2}=t^{7},  X_{3}=t^{10}, X_{4}=t^{15}. $$

\medskip
\noindent Recall that we can compute the Hilbert function of $C$
by two different methods. First we can compute the ideal of $C$
eliminating $t$ and then computing the Hilbert of ${\mathcal
O}_C$. In our setting we can also compute the Hilbert function of
$C$ using the fact $$\delta (C) = \# ( {\mathbb N} \setminus \Gamma
(C))$$

\noindent where $\Gamma (C) =<6,7,10,15>$ is the semi-group
generated by $C$. Hence we have $\delta (C) = 8$. On the other
hand we can desingularize $C$ by an unique Blow-up, so from
\cite{Kir59} we get that $\delta (C) = \rho = 8$. Hence we have
$h^1_C= 6T-8$.
Let us consider the family of parameterizations
$$\beta (U): X_{1}=t^{6}, X_{2}=t^{7}, X_{3}=t^{10}, X_{4}=t^{15}+
U t^{16}.$$
 \noindent Then we can consider the closed sub-scheme of
${({\bf k}^4,0)}_{S}$, $S=Spec({\bf k}[U])$, defined by $\beta
(U)$. From \cite{Tei77} we get that for all closed points $s \in
S$ the normalization of the fiber $Z_{s}$ is defined by $\beta
(s)$. Hence $\Gamma (Z_{s}) = \Gamma (C)$ and
$h^1_{Z_{s}}=h^1_{C}=6T-8;$
from \propref{critfam} we get that $Z \in {{\bf
H}_{4,6T-8}} (S)$.
\end{example}

\medskip
We end this section studying the relationship between families of
curve singularities and normally flat morphisms, see
\cite{Hir64a}.

\medskip
\begin{definition}
Let $S=Spec(A)$  be a scheme of finite type and let $Z$ be a closed
subscheme of ${({\bf k}^N,0)}_{S}$ such that there exists a closed
section $\sigma: S \longrightarrow  Z$ of $\pi:Z \longrightarrow S$.
We say that $Z$ is normally flat along $S$ if and only if
$Gr_{{{\mathcal I}}_{\sigma(S)}}({\mathcal O}_Z)$ is a flat
${\mathcal O}_S$ module.
\end{definition}

\medskip
\begin{proposition}
Let $Z$ be a closed subscheme  of ${({\bf k}^N,0)}_{S}$ such that
there exists a closed section $\sigma: S \longrightarrow  Z$ of
$\pi:Z \longrightarrow S$.

\noindent (1) if $Z$ is a normally flat scheme along $S$ and
verifies $(ii)'$ then $Z$ is a family of curves,

\noindent (2) if  $Z$ is a family of plane curve singularities
then $Z$ is normally flat along $S$.
\end{proposition}
\begin{proof}
(1) The result follows from the fact $Gr_{{{\mathcal
I}}_{\sigma(S)}}({\mathcal O}_Z)$
 is a flat ${\mathcal O}_S$-module
if and only if ${\mathcal O}_Z / {{\mathcal I}}_{\sigma(S)}^{n} =
{\mathcal O}_{Z_{n}}$ is a flat ${\mathcal O}_S$-module for all $n
\ge 1$.

\noindent (2) Let $Z$ be a family of plane curve singularities
over  an affine scheme $S=Spec(A)$. We need to prove that for all
closed point $s \in S$ and $n \ge 1$ the ${\mathcal
O}_{S,s}$-module ${\mathcal O}_{Z_{n}}$ is free of rank
$e_0n-e_0(e_0-1)/2$.
Let $m$ be the maximal ideal of $A$ defined by $s$. For all $n \ge
1$, we will prove that
$$
\frac{A_{m}[[X_{1},X_{2}]]}{I(Z)A_{m}[[X_{1},X_{2}]]+(X_{1},X_{2})^{n}}
$$

\noindent is a free $A_{m}$-module of rank $e_0n-e_0(e_0-1)/2$.
Let $F \in I(Z)A_{m}[[X_{1},X_{2}]] \subset A_m[[X_{1},X_{2}]]$ be a power series such
that $A_m[[X_{1},X_{2}]]/(F) \otimes_{A} {\bf k}$ is the local ring
${\mathcal O}_{Z_{s}}$.
Since the $\res$-vector spaces
$
(F)+(X_{1},X_{2})^{e_0+1} \subset
I(Z)A_{m}[[X_{1},X_{2}]]+(X_{1},X_{2})^{e_0+1}
$
have the same codimension
$e_0(e_0+1)-e_0(e_0-1)/2$, the vector spaces agree.
 From this it is easy to prove
(2).
\end{proof}

\medspace
\begin{remark}
The last Proposition enable us to consider the condition $(i)$ of
the definition of family of curve singularities as a weak form  of
normally flat morphism. Notice that the last three examples are in
fact normally flat families.
\end{remark}

\medspace
\begin{example}
Example of family not normally flat.
Let us consider the family of monomial curves
$$X_1=t^7, X_2=t^8, X_3=(1-u) t^9 + a t^{10}.$$
Since the singularity order and the Hilbert polynomial
does not depend on the parameter
$u$, from \cite{Tei77} and \propref{critfam}, we get that there exists
a family of curve singularities $\pi : Z\longrightarrow
S=Spec(\res[u])$ such that $Z_u$ is the monomial curve defined by
$u$.
On the other hand $H^1_{Z_0}(3)=5$, and $H^1_{Z_1}(3)=6$ so
$\pi$ is not a normally flat family.
\end{example}

\medskip
In \cite{Eli90} we defined rigid Hilbert polynomials as the
polynomials that determines the Hilbert function; i.e.
$p=e_0T-e_1$ is rigid if there exists a function $H_{p}:{\mathbb
N} \longrightarrow {\mathbb N} $ such that if $C$ is a curve
singularity with $h^1_{C}=p$ then $H^1_{C}=H_{p}$. For instance
$p=e_0T-e_1$ with $e_1 = e_0-1$, $e_0$, $e_0(e_0-1)/2-1$,
$e_0(e_0-1)/2$ are rigid polynomials, and any Hilbert polynomial
$p=e_0T-e_1$ with $e_0\le 5$ is rigid, see \cite{Eli90}. See also
\cite{EV91} for further results on rigid polynomials. Finally, it
is easy to prove

\medskip
\begin{proposition}
Every family of curve singularities with a rigid Hilbert
polynomial over a reduced base is normally flat.
\end{proposition}

\medskip
\section{Moduli space of curve singularities.}
\label{somesub}

The purpose of this section is to construct a moduli scheme
$\MS$   parameterizing the embedded curve singularities of $\EAN$
with fixed Hilbert  polynomial $p$, \thmref{monster}.

Let us recall that we assumed in the first section that $p=e_0
(T+1)-e_1$ is an admissible Hilbert polynomial, i.e. there exists a
curve singularity $C$ of $\EAN$ with Hilbert polynomial $p$. In
\cite{Eli90} we characterized  the admissible  Hilbert polynomials.
To recall that result we have to define some integers attached to
$e_0$ and the embedding dimension: given integers $1\le b\le e_0$ we
consider the following integers
$\rho_{0,b,e_0}=(r+1)e_0-\binom{r+b}{r}$, with $r$ the integer such
that $\binom{b+r-1}{r} \le e_0 < \binom{b+r}{r+1}$, and
$\rho_{1,b,e_0}=e_0(e_0-1)/2-(b-1)(b-2)/2$.

\medskip
\begin{proposition}
\label{compositio} There exists a curve singularity $C$ with
embedding dimension $b\le N$ and Hilbert  polynomial $p(T)=e_0
T-e_1$ if, and only if, either

\noindent
$(1)$ $b=1$, $e_0=1$, $e_1=0$, or

\noindent
$(2)$ $2 \le b \le e_0$, and  $\rho_{0,b,e_0} \le e_1 \le \rho_{1,b,e_0}$.

\noindent Moreover, for each triplet $(b, e_0, e_1)$ satisfying
the above conditions there is a reduced curve singularity $C
\subset \EAN$ with  embedding dimension $b$ and Hilbert
polynomial $p=e_0 T-e_1$, with $\res$ is an  algebraically closed
field.
\end{proposition}

\medskip
Let $F:{\mathbb  N} \longrightarrow {\mathbb N}$ be a numerical
function such that $F(t)\le b(t):=\binom{N+t-1}{N}$ for all $t\ge 0$.
For each $t \ge 0$ we denote by $G_{t}$ the Grassmannian
of $F(t)-$dimensional quotients of $R_t=R/M^{t}$, notice that
$R_t$ is a $b(t)$ dimensional $\res$-vector space.
Recall that
$G_{t}$ represents the contravariant functor
$\underline{G}_{t}:{\bf {Sch}} \longrightarrow  {\bf {Set}}$,
where ${\bf{Sch}}$ is the category of $\res-$schemes locally of
finite type, ${\bf{Set}}$ the category of sets, and
$\underline{G}_{t}(S)$ is the set of locally free quotients of
$\widetilde{R_{t}}_{(S)}$ of rank $F(t)$, see  \cite{EGA-I}-I-9.7.4.
 If $K$ is a
$F(t)-$dimensional quotient of $R_{t}$ then we will denote by
$[K]$ the corresponding closed point of $G_{t}$.

\medskip
We denote by $F_{r,n}$ the contravariant set-valued functor on
${\bf{Sch}}$ defined by: $F_{r,n}(S)$ is the set of $S-$ module
quotients ${\mathcal F}=\widetilde{R_{n}}_{(S)}/N$ such that the
${\mathcal O_S}\mbox{-module}$ $$ {\mathcal F}^{(i)}=\widetilde{
R_{n-i}}_{(S)}/ (\sigma _{n,i})_{*}(N) $$

\noindent belongs to $\underline{G}_{n-i}(S)$, $i=0,1,...,n-r$,
where $\sigma _{n,i}:\widetilde{ R_n} \longrightarrow \widetilde{
R_{n-i}}$ is the natural  morphism of sheaves. For all integers $r
\le n$, let $W(r,n,F)$ be  the  reduced subscheme of $G_{n}$ whose
closed points correspond to the $\res$-vector space quotients $R_n/E$ such
that $dim_{\res}(R_n/E+M^{t})=F(t)$ for all $t=r,\dots, n$.

\medskip
\begin{proposition}
\label{WWW-1} The scheme $W(r,n,F)$ represents the functor
$F_{r,n}.$
\end{proposition}
\begin{proof}
In order to prove the result we will use \cite{EGA-I}-0-4.5.4.
From the local nature of the definition of $F_{r,n}$ it is easy to
verify the second condition of \cite{EGA-I}-0-4.5.4. The first
condition follows from \cite{EGA-I}-I-9.7.4.6, and the third
condition from \cite{EGA-I}-I-9.7.4.7.

We will prove the fourth condition of \cite{EGA-I}-0-4.5.4. Let
$\{m_{i}\}_{i=1,...,b(n)}$ be a lexicographically ordered set of
monomials of ${{\bf k}[X_1,...,X_N]}$ such that their cosets in
$R_n$ form a ${\bf  k}-$basis.
 We denote by ${\mathcal H}$ the set of
$H=\{i_{1},...,i_{F(n)}\} \subset \{1,2,...,b(n)\}$ such that
$$card(H \cap \{1,2,...,b(n-i)\})=F(n-i)$$

\noindent for all $i=1,2,...,n-r.$ It is known that for every $H \in
\mathcal{H}$ we get an open set $B_{n}(H)$ of $G_{n}$: $B_{n}(H)$ is
the set of $F(n)$ dimensional vector spaces $R_n/L \quad$ such that
the cosets of $m_{i_{1}},...,m_{i_{F(n)}}$ in this quotient form a
${\bf  k}$-basis. This open set is isomorphic to the affine space of
matrices, $F(n) \times (b(n)-F(n))$.

Let $\varphi _{H}:{\mathcal O}_{Spec({\bf k })}^{F(n)}
\longrightarrow \widetilde{R_n}$ be the morphism of ${\mathcal
O}_{Spec({\bf  k})}-$modules defined by the ${\bf  k}-$linear map
$\varphi _{H}:{\bf k}^{F(n)}\longrightarrow R_n$
 with
$\varphi _{H}((a_{i})_{i \in H})= \sum_{i \in H}a_{i}m_{i}. $
By \cite{EGA-I}-0-4.5.4 we get  that $B_{n}(H)$ is
represented by the subfunctor $\underline{G}_{n,H}$ of
$\underline{G}_{n}$ defined by: $\underline{G}_{n,H}(S)$ is the
set of ${\mathcal F} \in \underline{G}_{n}(S)$ such that the
composition
$$ {\mathcal O}_S^{F(n)} \stackrel{\varphi
_{H}}{\longrightarrow} \widetilde{R_{n}}{(S)} \longrightarrow {\mathcal
F} $$ \noindent is an epimorphism.
It is easy to see that there
exists a one-to-one correspondence $\gamma $  between
$\underline{G}_{n,H}(S)$ and the set of ${\mathcal O_S}-$morphisms
$v:\widetilde{R_n} \longrightarrow {\mathcal O_S}^{F(n)}$ such that
$v\varphi_{H}=Id$; $v$ corresponds to ${\mathcal F}=\widetilde{
R_n}/Ker(v)$. Hence we have a exact sequence of set maps $$
\underline{G}_{n,H}(S) \stackrel{\gamma}{\longrightarrow}
 Hom_{{\mathcal O_S}}(\widetilde{R_n},{\mathcal O_S}^{F(n)})
{{{\scriptstyle \alpha_H} \atop {\rightrightarrows}}\atop
{\scriptstyle \beta_H} }
 Hom_{{\mathcal O_S}}({\mathcal O_S}^{F(n)},{\mathcal O_S}^{F(n)}),
 $$
\noindent with $\alpha_{H} (v)=v\varphi_{H}$,  $\beta_{H} (v)=Id;$
we get that $\underline{G}_{n,H}$ is representable by $B_{n}(H)$,
i.e. the kernel of the pair of morphisms
 $$
{\bf k}^{b(n)F(n)} {{{\scriptstyle \alpha_H} \atop
{\rightrightarrows}}\atop {\scriptstyle \beta_H} }
 {\bf  k}^{F(n)^{2}}.
$$

Let us consider the subfunctor $F_{r,n,H}$ of $F_{r,n}$ such that
for every $S$, $F_{r,n,H}(S)$ is the set of ${\mathcal F} \in
F_{r,n}(S)$ such that the composition $$ {\mathcal O}_{S}^{F(n)}
\stackrel{\varphi_H}{\longrightarrow}
 \widetilde{R_{n}}_{(S)}
\longrightarrow {\mathcal F} $$ \noindent is an epimorphism. Let
us consider the restriction of $\gamma $

$$\gamma :F_{r,n,H}(S) \longrightarrow
 Hom_{{\mathcal O_S}}(\widetilde{R_n},{\mathcal O_S}^{F(n)}),$$

\noindent first of all we need to compute the image of $\gamma $.
For this consider the projection in the first $F(n-i)$ components

$$\pi (i):{\mathcal O_S}^{F(n)} \longrightarrow {\mathcal
O_S}^{F(n-i)},$$

\noindent and the canonical monomorphism $$\sigma (i):
\widetilde{{M^{n-i}/M^{n}}} \longrightarrow \widetilde{ R_{n}}.$$

\noindent
 From \cite{EGA-I}-0-5.5.7 it is easy to prove

\medskip
\noindent {\bf CLAIM:} Let $v:\widetilde{R_{n}}_{(S)} \longrightarrow
{\mathcal O_S}^{F(n)}$ be a morphism such that $v\varphi_{H}=Id$,
then the sheaf $\widetilde{R_{n}}_{(S)}/Ker(v)$ is  locally free of rank
$F(n-i)$ if and only if $\pi (i)v\sigma (i)=0.$

\medskip
\noindent From the claim we can build up the following exact
sequence of map sets

$$F_{r,n,H}(S) \stackrel{\gamma}{\longrightarrow} Hom_{{\mathcal
O_S}}(\widetilde{R_n},{\mathcal O_S}^{F(n)}) $$ $$ {{{\scriptstyle
\varepsilon_{H}} \atop {\rightrightarrows}}\atop {\scriptstyle
\delta_{H} } }
 Hom_{{\mathcal O_S}}({\mathcal O_S}^{F(n)},{\mathcal O_S}^{F(n)}) \times
\prod_{i=1}^{n-r}Hom(\widetilde{M^{n-i}/M^{n}}_{(S)},{\mathcal
O_S}^{F(n-i)}), $$

\noindent with $\delta_{H} (v)=(Id;0,...,0)$, and
$\varepsilon_{H}(v)=(v\gamma_{H};\pi (i)v\sigma(i),i=1,2,...,n-r).$
From this and \cite{EGA-I}-I-9.4.9 we obtain that $F_{r,n,H}$ is
representable by the Kernel, say $X_{H}$, of the pair of morphisms
$$ {\bf k}^{b(n)F(n)} {{{\scriptstyle \varepsilon_{H}} \atop
{\rightrightarrows}}\atop {\scriptstyle \delta_{H} } } {\bf
k}^{F(n)^{2}} \times \prod_{i=1}^{n-r}{\bf  k}^{l(n,i)F(n-i)},
 $$
\noindent $l(n,i)=dim_{\bf k}(M^{n-i}/M^{n})$, so we get
\cite{EGA-I}-0-4.5.4(iv) for the family of functors
$\{F_{r,n,H}\}_{H \in \mathcal{H}}$. Hence $F_{r,n}$ is
representable by a scheme $X$, and $\{X_{H}\}_{H \in {\mathcal
H}}$, is an open cover of $X$. From the proof of the
representability of $X_{H}$ we deduce that $X_{H}$ is a linear
subspace of $B_{n}(H)$. From the definition of $W(r,n,F)$ and
\cite{EGA-I}-I-4.2.4(ii) we deduce that $X=W(r,n,F)$, so we get
that $W(r,n,F)$ represents the functor $F_{r,n}.$
\end{proof}

\medskip
We denote by $Hilb_{n}$ the Hilbert scheme parameterizing the closed
subschemes of $Spec(R_n)$ of length $F(n)$; $Hilb_{n}$ represents
the contravariant set-valued functor $\underline{Hilb}_{n}$ on
$\bf{Sch}$ for which $\underline{Hilb}_{n}(S)$ is the set of
morphisms $f:Z \subset Spec(R_n) \times S \longrightarrow S$
 with fibers of length $F(n)$, see \cite{Gro60}.
Let
$$a_{n+1}:W(r,n+1,F) \cap Hilb_{n+1} \longrightarrow
 W(r,n,F) \cap Hilb_{n}$$
be the morphism of schemes induced by the functorial morphism $$
a_{n+1}:F_{r,n+1} \times_{\underline{G}_{n+1}}
\underline{Hilb}_{n+1} \longrightarrow
F_{r,n}\times_{\underline{G}_{n}}
 \underline{Hilb}_{n},
$$

\noindent with $a_{n+1 (S)}({\mathcal F})={\mathcal F}^{(1)}$.

From now
on we assume that $F(t)=p(t-1)=e_0 t-e_1$, and for all $n \ge
e_0+1$ we put $W(n)=W(e_0+1,n,p)$ and $W'(n)= W(n) \cap
Hilb_{n}$.

\medskip
Let $C \subset {({\bf k}^N,0)}$ be a reduced curve singularity. We
will denote by $\delta (C)=dim_{{\bf k}}({\overline{\mathcal O}_C} /
{\mathcal O}_C)$ the order of singularity of $C$, here
${\overline{\mathcal O}_C}$ is the integral closure of ${\mathcal
O}_C$ on its full ring of fractions. We denote by $\mu (C)$ the
Milnor number of $C$; notice that $\mu (C)= 2 \delta (C)- r(C)+1$
where $r(C)$ is the number of branches of $C$, \cite{BG80}
Proposition 1.2.1.

\medskip
\begin{definition}
We
denote by $C_n(N,p)$  the set of points   of $G_n$
 defined by  all truncations $C_n$ where $C$ is a
curve singularity of Hilbert polynomial $h^1_C=p$.
\end{definition}

We will prove that $C_n(N,p)$ is in fact a constructible set of
$G_n$, see \propref{eles}.

\medskip
\begin{proposition}
\label{truncacio} (1) For all $n \ge e_0+1 $ it holds
$C_n(N,p) \subset  W'(n)$.

\noindent (2) Let $C$ be a curve singularity with Hilbert
polynomial $h^1_C=p$, then its tangent cone is determined by
$[C_n]$, $n \ge e_0+1 $.

\noindent (3) If $C$ is reduced then the analytic type of $C$ is
determined by $[C_{n}]$, $n \ge 2 \mu (C) +1$.
\end{proposition}
\begin{proof}
(1) We only need to prove that  for all curve singularity
of multiplicity $e_0$ it holds
$[C_n] \in  W'(n)$, $n \ge e_0+1 $.
Notice that the closed points of $W'(n)$ are the quotients
$R/I$ with $I \subset R$ ideal such that $M^{n} \subset I$ and
$dim_{{\bf k}}(R/M^{n-i}+I)=p(n-i-1)$ for $i=0,...,n-e-1$. In
particular if $C$ is a curve singularity of $({\bf k}^{N},0)$ with
Hilbert $p$ then $R/I(C)+M^{n}$ defines a closed point of
$W'(n)$ for all $n \ge e_0+1$, see \cite{Kir75}, Theorem 5.

\noindent (2) follows from \propref{jotatilde}.
From \cite{Eli86b}, Theorem 6, we deduce (3).
\end{proof}

\medskip
Notice that in the last proposition  we proved that $W'(n)$ contains all
$n$-th truncations of curve singularities with Hilbert polynomial
$p$. In order to take account of the Cohen-Macaulayness of the
curve singularities we need to shrink $W^{'}(n)$ by  intersecting
this scheme with some open Zariski subset of $G_{n}$. For this
end, given a linear form $L \in R$ we denote by $U_n(L)$ the
Zariski open set of $G_{n}$ whose closed points are the quotients
$R_{n}/E$ such that $dim_{\res}(R_{e_0+1}/(E,L)) \le e_0$, where
$(E,L)$ is the ideal generated by $E$ and $L$.

\medskip
\begin{proposition}
\label{CM} $(1)$ Let $I$ be an ideal of $R$ such that $A=R/I$
is a one-dimensional local ring of multiplicity $e_0$.
 Let $L$ be
an element of $R$, the following conditions hold
\begin{enumerate}
\item[(1)]
$dim_{\res}(A/LA) \ge e_0$ and we have equality if and only if $A$
is Cohen-Macaulay and the coset of $L$ in $A$ is a degree-one
superficial element.
\item[(2)]
If $dim_{\res}(R/I+M^{e_0+1})=p(e_0)$ then the following
conditions are equivalent:
\begin{enumerate}
\item[(a)]
        $[R/I+M^{e_0+1}]$ belongs to $U_{e_0+1}(L)$,
\item[(b)]
        $A$ is Cohen-Macaulay and the coset of $L$ in $A$ is a
        degree-one superficial element.
\end{enumerate}
\end{enumerate}

\noindent $(3)$ There exist linear forms $L_{1},...,L_{s}$ of $R$,
$s=e_0(N-1)+1$, such that for all curve singularity $C$ of $\EAN$
with Hilbert polynomial $p$  and $n \ge e_0+1$ it holds
$[C_{n}] \in W^{'}(n) \cap U_n(L_q)$.

\end{proposition}
\begin{proof}
\noindent $(1)$ Let us assume $dim_{\res}(A/LA) \le \infty$,
i.e. $L A$ is a $\max-$primary ideal of $A$. From \cite{Sal76},
Cap. I Proposition 3.4,  we get  $(1)$.
By $(1)$ we  deduce  that $(b)$ is  equivalent $dim_{\res}(R/I+(L))
\le e_0$. It is easy to see that this inequality  equivalent to
$dim_{\res}(R/I+M^{e_0+1}+(L)) \le e_0,$ i.e $[R/I+M^{e_0+1}]$
belongs to $U_{e_0+1}(L)$.

\noindent $(3)$ Let $C$ be a curve singularity $C$ of $\EAN$ with
Hilbert polynomial $p$. From \cite{ZS75II}, Chap. I Proposition
3.2, we deduce that if a linear form $L$ is a non zero-divisor in
$Gr({\mathcal O}_C)_{red}$ then $L$ is a degree one superficial
element of ${\mathcal O}_C$. Let $L_{1},...,L_{s}$, $s=e_0(N-1)+1$
be linear forms such that any subset of $N-1$ elements is
$\res-$independent. From this and the $(1)$ it is easy to deduce
the claim.
\end{proof}

\medskip
We will define a sub-scheme $\Xi_n$ of $Hilb_n$
 taking account of the condition (2) of the definition of
 $\mathbb T_n$.
  For this end we have to define some
special cells of the Grassmannian $G_n$, from a deep study of these
cells we will deduce that the morphism $a_n$ is affine for a big
enough $n$, \propref{fibration} $(1)$.

Let ${\bf Mon}=\{m_{1},m_{2},...\}$ be  the set of monomials of $S$
ordered with respect to the degree-lexicographic ordering. Given
multi-indexes $i_{.}=\{i_{1},...,i_{p(e_0-1)}\} \subset
\{1,2,...,b(e_0)\}$, $j_{.}=\{j_{1},...,j_{e_0}\} \subset
\{b(e_0)+1,...,b(e_0+1)\}$, and a linear form $L_q$, we define
$D_{n}(i_{.},j_{.},q)$ as the linear subspace of $R_n$ generated
by the $p(n-1)$ linear independent monomials
$$ m_{i_{1}},...,m_{i_{p(e_0-1)}};
L_q^{r}m_{j_{1}},...,L_q^{r}m_{j_{e_0}}, \quad r=0,...,n-e_0-1. $$

\noindent For all $n \ge e_0$ we fix  a $\res$-basis
$V_{n}(i_{.},j_{.},L)  $ of $R_n$ adding to these elements
monomials of suitable degree. We denote by $B_{n}(i_{.},j_{.},q)$
the Zariski open subset of $G_{n}$ with closed points $[R_n/E]$
such that  the projection $D_{n}(i_{.},j_{.},q) \longrightarrow
R_n/E$ is an isomorphism.

For all triplet $i_{.},j_{.},q$, we consider the open sets
$B^{'}_n(i_{.},j_{.},q)= B_n(i_{.},j_{.},q) \cap U(L_q)$; we set
$B^{'}_n=\cup_{i_{.},j_{.},q} B^{'}_n(i_{.},j_{.},q)$. We denote
by $\Xi_{n}$ the open sub-${\bf k}-$scheme of $W^{'}(n)$
$$\Xi_{n}= W^{'}(n) \cap B^{'}_n $$

\noindent and $\Xi _{n}(i_{.},j_{.},q)= \Xi_n \cap
B^{'}_{n}(i_{.},j_{.},q).$ Notice that the morphism $a_n:W^{'}(n)
\longrightarrow W^{'}(n-1)$ induces morphisms
$a_n:\Xi_{n}(i_{.},j_{.},q) \longrightarrow
\Xi_{n-1}(i_{.},j_{.},q)$, $a_n:\Xi_{n} \longrightarrow
\Xi_{n-1}$, for all $n \ge e_0+1$.

\medskip
\begin{proposition}
\label{eles} $(1)$ For all curve singularities  $C$ of $\EAN$ with
Hilbert polynomial $p$ there exists indexes  $i_{.},j_{.},q$ such
that $[C_{n}] \in \Xi_{n}(i_{.},j_{.},q)$, for all $n \ge e_0+1$.

\noindent $(2)$
$C_n(N,p)$ is a
 constructible set and $C_n(N,p) \subset \Xi_n$,

\noindent $(3)$
For all $i_{.},j_{.},q$ it holds
$a_n^{-1}(\Xi_{n-1}(i_{.},j_{.},q)_{red})=
\Xi_{n}(i_{.},j_{.},q)_{red}.$
\end{proposition}
\begin{proof}
$(1)$
Follows from \propref{nec-cond} $(2)$ and \propref{CM}.

\noindent $(2)$
From \thmref{big-lifting} we get that
$C_n(N,p) =a_{\beta(n,e_0)}\dots a_{n+1}(\Xi_{\beta(n,e_0)})$,
so $C_n(N,p)$ is a constructible set contained in $\Xi_n$.

\noindent $(3)$
Let $x=[R/I+M^{n}]$ be a closed point of $\Xi_{n}$ such that $a_n(x)$
belongs to $\Xi_{n-1}(i_{.},j_{.},q)_{red}$.
Let $\maxn$ be the maximal ideal of $R/I+M^{n}$.
Since $a_n(x)$ belongs to $B_{n-1}(i_{.},j_{.},q)$
we have that
$dim_{\res}(\maxn^{e_0+s}/\maxn^{e_0+s+1})=e_0$ for $s=1,\dots,
n-e_0-1$, so we have
$$
\frac{\maxn^{e_0-1}}{\maxn^{e_0}} \stackrel{.L_q}{\longrightarrow}
\frac{\maxn^{e+s-1}}{\maxn^{e+s+1}}
$$

\noindent
is an epimorphism for $s=1,\dots, n-e_0-1$.
Since $a_n(x)\in W(n)$ we get that these morphisms are in fact
isomorphism, so  $x \in \Xi_{n}(i_{.},j_{.},q).$
\end{proof}

\medskip
Let $X, Y$ be $\res$-schemes. Given constructible sets $A \subset
X$, $B \subset Y$  we say that a map $f: A \longrightarrow B$ is a
piecewise trivial fibration with fiber a   $\res-$scheme $F$
 if there exists a
finite partition of $B$ in locally closed subsets $S \subset Y$
such that $f: f^{-1}(S) \longrightarrow S$ is a fibration of fiber
$F$, see  \cite{DL99b}.

The key result in the definition of a local motivic integration in the moduli space $\MS$,
\thmref{monster}, is the following result:

\medskip
\begin{proposition}
\label{fibration}
$(1)$
For all $n \ge  e_0 +4$ the morphism
$a_n:\Xi_n \longrightarrow \Xi_{n-1}$ is  affine.

\noindent $(2)$ Given integers $\nu \ge 2$,  $n \ge \delta(N, e_0)$,
there exists a constructible set $\Sigma_{\nu}^n \subset C_n(N,p)$
such that: for all curve singularities  $C$ with Hilbert polynomial
$p$, $[R/I(C)+M^n]$ is a closed point of $\Sigma_{\nu}^n$ if and
only the minimal numbers of generators of $I(C)$  is $\nu$.

\noindent
$(3)$ If $\nu =N-1$, i.e. the case of complete intersection singularities,
then we set $\Sigma_{ci}^n:=\Sigma_{N-1}^{n}$ and
the restriction of $a_{n}$
$$
a_{n}: \Sigma_{ci}^{n} \longrightarrow \Sigma_{ci}^{n-1}
$$

\noindent
is an exhaustive  piecewise fibration with fiber  $F \cong \res^{(N-1) e_0}$,
$n \ge \delta(N, e_0)$.
\end{proposition}
\begin{proof}$(1)$ We will perform the proof in three steps.

\noindent
{\bf Step 1}
For all $n \ge e_0+4$ the  cell $B_{n}(i_{.},j_{.},q)$ is
isomorphic to an affine space ${\bf k}^{s_{n}}$, for some integer $s_n$, and
$a_n$ restricts to an affine morphism
$a_n:W(n) \cap B_{n}(i_{.},j_{.},q) \longrightarrow W(n-1)
\cap B_{n-1}(i_{.},j_{.},q)$.

\vskip 1mm
\noindent
{\bf Proof:}
It is easy to see that considering the base $V_{n}(i.,j.,q)$
that the open set
$B_{n}(i_{.},j_{.},q)$ is isomorphic to the affine space
of $(b(n)-p(n)) \times b(n)$ matrices, and that
$W(n) \cap B_{n}(i_{.},j_{.},q)$ is a linear subspace
of $B_{n}(i_{.},j_{.},q)$.
A general element of $W(n) \cap B_{n}(i_{.},j_{.},q)$
 can be written as follows
$$
U=
\begin{pmatrix}
\begin{tabular}[t]{cccc}
\parbox[b]{14mm}{
\fbox{\parbox[b]{14mm}{
\begin{center} \quad \\  Id \\ \quad \\ \end{center} }}
\fbox{\parbox[b]{14mm}{ \begin{center} 0  \end{center} }} }
 &
\parbox[b]{13mm}{
\fbox{\parbox[b]{13mm}{ \begin{center}\quad \\ A \\ \quad
\\ \end{center} }} \fbox{\parbox[b]{13mm}{\begin{center} 0
\end{center}}} }
&
\parbox[b]{7mm}{
\fbox{\parbox[b]{7mm}{ \begin{center} \quad \\ 0 \\ \quad
\\ \end{center} }}
\fbox{\parbox[b]{7mm}{ \begin{center} Id  \end{center} }} } &
\parbox[b]{7mm}{
\fbox{\parbox[b]{7mm}{\begin{center} \quad \\ B \\ \quad \\
\end{center} }}
\fbox{\parbox[b]{7mm}{\begin{center} C \end{center} }} }
\end{tabular} \; \;
\end{pmatrix}
$$

\noindent
Hence for all $n \ge e_0+4$ we get
$B_{n}(i_{.},j_{.},q) \cong {\bf k}^{s _{n}}.$
From this it is easy to see that $a_n$ is a linear
projection between ${\bf  k}^{s _{n}}$ and
${\bf  k}^{s _{n-1}}$ where $a_n(U)$ is the matrix obtained from
$U$ deleting its last row and the last two columns.
Hence we get {\bf Step 1}.

\medskip
\noindent
{\bf Step 2}
 For all $n \ge e_0+4$,
$a_n:\Xi_{n}(i_{.},j_{.},q) \longrightarrow
\Xi_{n-1}(i_{.},j_{.},q)$ is an  affine morphism.
\vskip 1mm
\noindent
{\bf Proof:}
Let us consider the restriction of $a_{n}$
$$a_n:W(n) \cap B_{n}(i_{.},j_{.},q)) \longrightarrow W(n-1)
\cap B_{n-1}(i_{.},j_{.},q).$$

\noindent
Since
$W(r) \cap B^{'}_{r}(i_{.},j_{.},q)$ is an open subset of
$W(r) \cap B_{r}(i_{.},j_{.},q)$ for $r=n,n-1$, and
$$a_n^{-1}(W(n-1) \cap B^{'}_{n-1}(i_{.},j_{.},q))= W(n) \cap
B^{'}_{n}(i_{.},j_{.},q),$$

\noindent
from \cite{EGA-I}-I-9.1.2 we deduce that
$$a_n:W(n) \cap B^{'}_{n}(i_{.},j_{.},q)) \longrightarrow W(n-1)
\cap B^{'}_{n-1}(i_{.},j_{.},q)$$

\noindent
is affine.
Since $\Xi_{r}(i_{.},j_{.},q)$ is a closed subscheme of
$W(r) \cap B^{'}_{r}(i_{.},j_{.},q)$, for $r=n,n-1$,
by \cite{EGA-I}-I-9.1.16(i),(v) we obtain {\bf Step 2}.

\medskip
\noindent
{\bf Step 3}
For all $n \ge e_0+4$ the morphism
$a_{n}:\Xi_{n} \longrightarrow \Xi_{n-1}$
is affine.

\vskip 1mm
\noindent
{\bf Proof:}
By {\bf Step 2}, \propref{eles} (3), and \cite{EGA-I}-I-9.1.18 we get
that for all $n \ge e_0+4$ the morphism $a_n:\Xi_{n} \longrightarrow
\Xi_{n-1}$ is affine.

\medskip
\noindent
$(2)$ We denote by $\nu(B)=dim_{\res}(B)$ the minimal number of
generators of a finitely generated $R-$module $B$.
Let $G_{\nu}^n$ be the constructible sub-set of $\Xi_n$
whose closed points $[R/J]$ verifies $\nu(J)=\nu$.
We define $\Sigma^n_{\nu}=G_{\nu}^n \cap C_n(N,p)$, \propref{eles}.

Let $C$ be a curve singularity with Hilbert polynomial $p$.
From \cite{RV80} and \propref{jotatilde} $(3)$ we get
$I(C) \cap M^n \subset  I(C) M$, $n \ge \delta(N,e_0)$, so we have
$
\nu(I(C))= \nu\left( \frac{I(C)+M^n}{M^n}\right)
$
and $[R/I(C)+M^n]$ is a point of $\Sigma^n_{\nu}$ if and only
if $\nu = \nu(I(C))$.
Moreover, let $f_1,\dots,f_{\nu}$ be elements $I(C)$
such that their cosets in $ I(C)+M^n/M^n$ form
a minimal system of generators, then $f_1,\dots,f_{\nu}$
is minimal system of generators of $I(C)$.

\medskip
\noindent
$(3)$
We set  $\Sigma_{ci}^n:=\Sigma^n_{N-1}$. i.e. $ \nu=N-1$.
Let $x=[R/J+M^n]$ be a closed point of $\Sigma_{ci}^n$, $J=I(C)$ with $C$ a curve singularity, and
$y=[R/J+M^{n-1}]=a_{n}(x)$  its image.
We may assume that
$y \in \Xi_{n-1}(i_{.},j_{.},q)_{red},$
for some set of indexes $i_{.},j_{.},q$.
Let $U$ be the associated matrix to $x$, see {\bf Step 1}.
Let ${\bf f}=f_1,\dots,f_{\nu}$ be $\nu$ rows of $U$ such that their
cosets in $R_{n-1}$ form a minimal system of generators of $J+M^{n-1}$, see $(2)$.
Notice that ${\bf f}$ is a minimal system of generators
of $J+M^n$ and that all entries of $U$ are determined by  ${\bf f}$.

If $x'=[R/J'+M^n]$ is a closed point of $\Sigma_{ci}^n$ such that
$a_n(x')=a_n(x)$ then $J'$ admits a minimal system of generators
${\bf g}=g_1,\dots,g_{\nu}$ such that $f_i=g_i$ modulo $M^{n-1}$.
From \propref{eles} (3) $x'$  belongs to $\Xi_{n}(i_{.},j_{.},q)$
and then defines a matrix $U'$ with the same shape of $U$.
Hence ${\bf f}-{\bf g}$ is a set of $\nu$ homogeneous  polynomials
of degree $n$ belonging to $D_{n}(i_{.},j_{.},L_q)$, i.e.
the monomials corresponding to the matrix $B$.
Then we have that
$\Sigma_{ci}^n \subset \Sigma_{ci}^{n-1} \times \res^{(N-1) e_0}$.

Let $z$ be a closed point of $ \Sigma_{ci}^{n-1} \times \res^{(N-1) e_0}$
such that $a_n(z)=x$.
Let $J_{\varepsilon}$ be the ideal of $R$ generated by
$f_1+\varepsilon_1,\dots ,f_{\nu}+\varepsilon_{\nu}$
,where $\varepsilon=\varepsilon_1,\dots ,\varepsilon_{\nu}$ is a set of
$\nu$ homogeneous  polynomials
of degree $n$ belonging to $D_{n}(i_{.},j_{.},L_q)$, such that $z=[R/J_{\varepsilon}+M^n]$.
From the condition $(2)$ of the definition of $\mathbb T_n$ we deduce that $dim(R/J_\varepsilon)\le 1$;
since $\nu=N-1$ we get that $dim(R/J_\varepsilon)= 1$ and then $C=Spec(R/J_\varepsilon)$ is a curve singularity.
From the  definition of  $\mathbb T_n$ we deduce that  $C$ is a curve singularity with  Hilbert polynomial $p$,
so $z\in \Sigma_{ci}^n$.
\end{proof}

\medskip
\medskip
From \propref{fibration} and \cite{EGA-IV}-IV-8.2.3, see also
\cite{SGA-4} exposse VII,   we deduce that the inverse system
$\{\Xi_{n},a_n\}_{n \ge e_0+1}$ has a limit $\MS$ that we describe
as follows.
 Since the
maps $a_{n}$ are affine, we have $\Xi _{n} = Spec({\mathcal A}_{n})$
where ${\mathcal A}_{n}$ is a quasi-coherent sheaf of ${\bf
k}$-algebras over $\Xi _{n}$, \cite{EGA-II} 1.3.7. Then we define
$\MS = Spec({\mathcal A})$, with $${\mathcal A}= \ilim {\mathcal
A}_{n},$$

\noindent see \cite{SGA-4} exposse VII. In particular we get that
for all point $x$ of $\MS$ it holds $$ {\mathcal O}_{\MS , x} \cong
\ilim \pi ^{*}_{n}({\mathcal O}_{\Xi _{n}, \pi _{n}(x)}), $$
\noindent \cite{EGA-IV}, 8.2.12.1. Where we have denoted by

$$\pi _{n}:\MS \longrightarrow \Xi_{n},$$

\noindent the natural projection, $n \ge e_0+1$. Given $i  >  j \ge
e_0+1$ we  define the affine map $a_{i,j}:\Xi_i\longrightarrow
\Xi_j$ by the composition $a_{i,j}=a_i a_{i-1} \cdots a_{j+1}$.

\medskip
\begin{theorem}
\label{monster} The scheme $\MS$ pro-represents the functor $\MF$.
\end{theorem}
\begin{proof}
We will prove that the functor $\MF$ is isomorphic to
$h=Hom(\cdot,{\bf{H}}_{N,p(T)})$.
 Let $S$ be an object of ${\bf{Aff}}$
and let $Z \subset {({\bf k}^N,0)}_{S}$ be a family of curves over
$S$ with Hilbert polynomial $p$. Since $Z_{n}$ is a flat scheme
over  $S$ with fibers of length $p(n)$, we have ${\mathcal
O}_{Z_{n}}  \in \underline{Hilb}_{n}(S)$. On the other hand, if
$F_{e_0+1,n}$ is the functor that is represented by $W(n)$ then
${\mathcal O }_{Z_{n}}\in F_{e_0+1,n}(S)$, and $$ {\mathcal O
}_{Z_{n}}\in (F_{e_0+1,n}
\times_{\underline{G}_{n}}\underline{Hilb}_{n})(S).
 $$

\noindent Hence by \propref{WWW-1} there exists a morphism $\sigma
_{n}(Z):S \longrightarrow W^{'}(n)$ for all $n \ge e_0+1$. Recall
that $B'_{n}$ is an open subset of $G_{n}$, so $\sigma _{n}(Z)$
factorizes through $X_{n}=W^{'}(n) \cap B'_{n}$
if and only if for all closed point $s \in S$, we have
$\sigma _{n}(Z)(s) \in B^{'}_{n}$. Since $\sigma
_{n}(Z)(s)=[(Z_{s})_{n}]$
 and $Z$ is a curve singularity with Hilbert polynomial $p(T)$
by \propref{eles} we get $\sigma _{n}(Z)(s) \in B^{'}_{n}$.
Hence
we have a morphism $$\sigma _{n}(Z):S \longrightarrow \Xi_{n}$$

\noindent for all $n \ge e_0+1$.
It is easy to see that for all $n
\ge e_0+2$ it holds $a_n \sigma _{n}(Z)=\sigma _{n-1}(Z)$
so we have a morphism  $\sigma _{*}(Z):S \longrightarrow {\bf{H}}_{N,p(T)}$.
From this we get a functorial morphism $$ \sigma
:\underline{\bf{H}}_{N,p(T)} \longrightarrow h, $$ \noindent
sending $Z \in \MF (S)$ to  $\sigma _{*}(Z)$.

To complete the proof we need to prove that $\sigma(S) $ is
bijective for all $S$. The injectivity is straightforward. Let
$g:S=Spec(A) \longrightarrow \MS$ a morphism of ${\bf k}$-schemes.
The morphism $$g_{n}=\pi _{n}g:S=Spec(A) \longrightarrow \Xi_{n}
\subset {Hilb}_{n}$$

\noindent
 defines an ideal $J_{n} \subset A[[\underline X]]/(X)^{n}$,
the compatibility relations $a_n g_{n}=g_{n-1}$ give us
$J_{n}+(\underline X)^{n-1}=J_{n-1}$ for all $n \ge e_0+2$. If we write $J=
\cap_{n\ge e_0+1}J_{n}$ then it holds $J_{n}=J+(\underline X)^{n}$ for all $n
\ge e_0+1$. From this we get $A[[\underline X]]/J $ is isomorphic to the
limit of the inverse system defined by $A[[\underline X]]/J_{n}$, $n \ge
e_0+1$.

Let us consider the scheme $Z\subset {({\bf k}^N,0)}_{S}$ defined
by the ideal $J$. We will prove that $Z$  is a family and $\sigma
(Z)=g$. Since $I(Z_{n})=J_{n}$ we have that $Z$ verifies condition
(i) of the definition of a family. Let $s$ be a closed point of
$S$, we have to prove that $Z_{s}$ is a curve singularity with
Hilbert polynomial $p(T)$. Notice that for all $n \ge e_0+1$ it
holds $(Z_{s})_{n}=(Z_{n})_{s}$
so $Z_{s}$ is a one-dimensional sub-scheme of ${({\bf
k}^N,0)}$ with Hilbert polynomial $p$. From $\Xi_{n} \subset
B_{n}$ and Proposition 2.3 we get that $Z$  is a family. Since
$\sigma (Z)=g$ we obtain the theorem.
\end{proof}

\medskip
\begin{remark}
\label{projcons}
Notice that from the last result and \propref{eles} we get that
$\pi_n(\MS)=C_n(N,p)$ is the constructible set of all $n$-truncations
of curve singularities with Hilbert polynomial $p$.
\end{remark}

On the other hand, notice that $\MS$ is not a ${\bf k}$-scheme locally of finite
type. Hence from the previous result we cannot deduce the
existence of a universal family. In the next result we will
construct a universal family for the scheme $\MS$.

\medskip
\begin{theorem}
\label{universal} There exists a ${\bf k}$-scheme ${\bf Z}_{N,p}$,
limit of an inverse system $\{U_{n}, \alpha _{n}\}_{n \ge e_0+1}$,
and a morphism $\varphi :{\bf Z}_{N,p} \longrightarrow \MS$ such
that for all families of curve singularities $f: Z \longrightarrow
S$ with Hilbert polynomial $p$  there exists a unique morphism
$\sigma : S \longrightarrow \MS$ such that $Z_{n} \cong S \times
_{\Xi_{n}} U_{n}$, $n \ge e_0+1$.
\end{theorem}
\begin{proof}
Let $\rho_{n}:T_{n}  \longrightarrow Hilb_{n}$ be the universal
family of $Hilb_{n}$, $n \ge e_0+1$.
 Recall that $T_{n}$ is a
closed sub-scheme of $Hilb_{n} \times Spec(R_{n})$, and that
$\rho_{n}$ is the restriction of the projection to the first
component. Hence $\rho_{n}$ is an affine morphism.

From the definition of $\Xi _{n}$ we have that $\Xi _{n}$ is an
open subs-cheme of $Hilb_{n}$.
If we denote by $U_{n}$ the fibred
product $$U_{n}= T_{n} \times_{Hilb_{n}} \Xi_{n}$$

\noindent
 we get that
the induced morphism $\rho_{n}: U_{n}  \longrightarrow \Xi_{n}$ is
affine. Let us consider the following commutative diagram

$$
\xymatrix{
U_{n+1}    \ar[d]_{\relax \rho_{n+1}}   \ar[r]^{\relax \alpha _{n}} &    U_{n} \ar[d]^{\relax \rho_{n}}\\
\Xi_{n+1}  \ar[r]_{a_{n}}       & \Xi_{n}
}
$$

\noindent where $\alpha _{n}$ is the morphism obtained from the
universal family $T_{n}  \longrightarrow Hilb_{n}$ and the fact
that $\Xi_{n}$ is an open subset of $Hilb_{n}$.
Since $a_{n} \rho
_{n+1}$ and $\rho _{n}$ are affine we get that $\alpha _{n}$ is
also affine, \cite{EGA-I}-I-9.1.16(v).
We define ${\bf Z}_{N,p}$ as the limit of the inverse system $\{ U_{n}, \alpha
_{n}\}_{n \ge e_0+1}$.
 We denote by $$\sigma :{\bf Z}_{N,p}
\longrightarrow \MS$$

\noindent
 the morphism induced by the morphism of inverse systems
$\{\rho_{n}\}: \{ U_{n}, \alpha _{n}\}  \longrightarrow \{
\Xi_{n}, a_{n}\}$.
We leave  to the reader the proof of the universal property
of $\sigma :{\bf Z}_{N,p}  \longrightarrow \MS$.
 \end{proof}

\medskip
In the following result we will prove that $\MS$ has finite
cohomological dimension. Recall that if $Y$ is a scheme, the
cohomological dimension $cd(Y)$ of $Y$ is the least integer $i$
such that $H^{j}(Y,{\mathcal F})=0$ for all quasi-coherent sheaves
${\mathcal F}$  and $j > i$.

\medskip
\begin{proposition}
\label{iarrobino}
There exists a constant $g(N)$ such that
$$
cd(\MS)  \le  g(N) p(e_0+3)^{2-2/N}.
$$
\end{proposition}
\begin{proof}
Let ${\mathcal F}$ be a quasi-coherent sheaf of $\MS$ and let $j
\ge cd(\Xi_{e_0 +3})$.
The inverse system $\{\Xi_{n},a_{n}\}_{n\ge e_0+1}$
defines a direct system of groups $\{H^{j}(\Xi_{n},\pi
_{n*}({\mathcal F})),a_{n*}\}_{n \ge e_0 +1}.$
 From \cite{SGA-4}, Exp. VII Corollaire 5.10, we
have
 $$
H^{j}(\MS,{\mathcal F}) \cong
\lim_{\stackrel{\longrightarrow}{n}}
H^{j}(\Xi_{n},\pi _{n*}({\mathcal F})). $$ \noindent Since $a_{n}$
is an affine morphism  for all $n \ge e_0+4$ we get $$
H^{j}(\Xi_{n},\pi _{n*}({\mathcal F})) \cong H^{j}(\Xi_{e_0+3},\pi
_{e_0+3*}({\mathcal F}))=0.$$
Hence we have $cd(\MS) \le cd(\Xi _{e_0+3}).$
Since $\Xi _{e_0+3}$ is a sub-scheme of $Hilb_{e_0 +3}$ we get
$$cd(\Xi _{e_0+3}) \le dim(\Xi _{e_0+3}) \le dim (Hilb_{e_0 +3}).$$
The claim  follows from  \cite{BI78}.
\end{proof}

\medskip
\begin{proposition}
There exists a one-to-one correspondence between the set of curve
singularities with Hilbert polynomial $p$ and the set of
rational points of ${\bf{H}}_{N,p}$. Moreover, if the
cardinality of $\res$ is strictly greater than the cardinality of
$\mathbb N$ then every closed point of ${\bf{H}}_{N,p}$ is rational.
\end{proposition}
\begin{proof}
The first part follows from the \propref{monster} and
\cite{EGA-I}-I-3.3.5. By construction we have that
${\bf{H}}_{N,p}$ can be covered by open affine sets $Spec(A)$
where $A$ is a $\res-$algebra countably generated. From
\cite{Gil71}, Proposition 2.6, we obtain the claim.
\end{proof}

\medskip
From now on we assume that the cardinality of the ground field
${\bf k}$ is greater  than the cardinality of $\mathbb N$.
Hence we have a one-to-one correspondence between the set
of curve singularities with Hilbert polynomial $p$ and the set
of closed points of ${\bf{H}}_{N,p}$.

\begin{definition}
Let $x$ be a closed point of ${\bf{H}}_{N,p}$, we will denote
by $C_{x}$ the curve singularity defined by $x$, and by
$I_{x}=I(C_{x}) \subset R={\bf k}[[X_{1},...,X_{N}]]$ the ideal associated to $C$.
The maximal ideal of ${\mathcal O}_{C_{x}}=R/I_{x}$ is
$\bf{m}_{x}$.
\end{definition}

\medskip
Given a function $F: {\mathbb N} \longrightarrow {\mathbb N}$, we
say that $F$ is admissible for the polynomial $p(T)$ if $F(t)=p(t)$
for $t \ge e_0-1$.
 Given an admissible function $F$ for $p$ and an
integer $r$, $1 \le r \le e_0$, we define a contravariant set valued
functor on ${\bf Aff}$ such that for all ${\bf k}$-scheme of finite
type $S$ we have

$$
{{\underline{\bf H}}_{N,F,r}}(S)= \left\{
 \hskip 1mm
\parbox{8.2cm}{
$Z \in \MS (S)$ such that for all $n=r,...,e_0+1$ the morphism

\centerline{$\pi : Z_{n} \longrightarrow
S$}

is flat with fibers of length $F(n)$. } \hskip 1mm \right\} $$

\medskip
\noindent
Notice that $\MF = {{\underline{\bf H}}_{N,F,e_0+1}}$, and
that ${{\underline{\bf H}}_{N,F,1}}(S)$ is the set of normally
flat families of base $S$. We denote by ${{\bf H}_{N,F,r}}$ the
${\bf k}$-scheme ${{\bf H}_{N,F,r}}= W(r,e_0+1,F)
\times_{G_{e_0+1}} \MS.$
Notice that  ${{\bf H}_{N,F,r}}$ is a closed sub-scheme
of $\MS$. From the \thmref{monster} it is easy to prove

\medskip
\begin{proposition}
${{\bf H}_{N,F,r}}$ represents the functor ${{\underline{\bf
H}}_{N,F,r}}$.
\end{proposition}

\medskip
\noindent Given an admissible function $F$ for $p$ there is a family
of closed sub-schemes of $\MS$ $$ {{\bf H}_{N,F,1}} \subset {{\bf
H}_{N,F,2}} \subset \dots \subset {{\bf H}_{N,F,e_0+1}} = \MS. $$ We
say that ${{\bf H}_{N,F,r}}$ is the Hilbert stratum  of $\MS$ with
respect to $(F,r)$. Recall that it is an open problem to
characterize the admissible Hilbert functions, we only have
characterized the asymptotically behavior  of Hilbert functions,
i.e. Hilbert polynomials \propref{compositio}.

 Very few  properties of
Hilbert functions are known, see for instance \cite{Eli93a}.
If $p$ is a rigid polynomial we know that there exists a unique
admissible function $F$ associated to $p$, in this case we have
$({{\bf H}_{N,F,r}})_{red} =({{\bf H}_{N,F,r+1}})_{red}$,
 for all $r=1,...,e_0$.



\bigskip
\bigskip
We say that a subset $D$ of $\MS$ is a cylinder if $D=
\pi_n^{-1}(D_n)$ where $D_n \subset \pi_n(\MS)$ is a constructible
set  of $\Xi_n$ for some $n\in \mathbb N$. We denote by $n(D)$ the
least integer $n$ verifying such a condition. Notice that
$D_n=\pi_n(D)$ for  $n \ge n(D)$.

In the next result we will prove for $\MS$ some results of
\cite{Bat98}, Proposition 6.5, 6.6, and  \cite{DL99b}, Lemma 2.4,
proved for the jet schemes.

\medskip
\begin{proposition}
\label{cylinder}
{\rm (1)} The collection  $\mathcal C _{N, p}$ of cylinders
of $\MS$ is a Boolean  algebra of sets.

\noindent {\rm (2)} Given a cylinder $D$ and a family of cylinders
$\{D_{i}\}_{i \in \mathbb N}$ of $\MS$ such that $D= \cup_{i \in
\mathbb N} D_i$ there exist a finite set of indexes $K \subset
\mathbb N$ such that $D= \cup_{i \in K} D_i.$
\end{proposition}
\begin{proof}
$(1)$ Since $\pi_n(\MS)=C_n(N,p)$ are constructible sets it is easy to prove $(1)$, \remref{projcons}.
$(2)$ Let us consider the cylindrical sets $Z_n=D
\setminus (D_1 \cup \dots \cup D_n)$, $n\ge 0$. Since $Z_1 \supset
Z_2 \supset \dots$ we have that the claim  is equivalent to
$Z_{i_0} = \emptyset$ for some index $i_0$.
We can assume that  there exists  an increasing set of integers
$\{n_i\}_{i \ge 1}$, $n_i \ge n(Z_i)$, such  that
 $\pi_{n_i}(Z_i)=T_{n_i}$, with $T_{n_i}$ a constructible set of
 $\Xi_{n_i}$,
 and $T_{n_{i+1}} \subset a^{-1}_{n_{i+1},n_i}(T_{n_i})$, $i\ge 1$.
Since $\cap_{i \ge 1} Z_i = \emptyset$ from \cite{EGA-IV},
Proposition 8.3.3,  we have  that $Z_{i_0} = \emptyset$ for some
index $i_0$, and then we get  $(2)$.
\end{proof}

\medskip
We denote by $K_0({\bf Sch})$ the Grothendieck ring of ${\bf Sch}$.
Let $ \mathbb L =[\res]$ be the coset of $\res$ in $ K_0({\bf
Sch})$, we set ${\mathcal M}=K_0({\bf Sch})[\, \mathbb L^{-1}]$. Let
$\mathcal F =\{F^n \mathcal M\}_{n \in \mathbb Z}$ be Kontsevich's
filtration of $\mathcal M$: $F^n \mathcal M$ is the sub-group of
$\mathcal M$ generated by the elements of  the form $[V] \mathbb L
^{-i}$ for $i- dim(V) \ge n$. We denote by $\widehat{ \mathcal M}$
the completion of $\mathcal M$ with respect the filtration $\mathcal
F$.

Let $C$ be a cylinder of $\MS$, we say that $C$ is $c$-stable at
level $n \ge n(C)$ if
$$[\pi_{n+1}(C)]=[\pi_{n}(C)] \; {\mathbb L}^c \in K_0({\bf Sch}),$$
$C$ is
$c$-stable if there exists an integer $n_0 \ge n(C)$ such that $C$
is $c$-stable at level $n$ for all $n\ge n_0$. We denote by
$sn(C)$ the least integer $n_0$ verifying such a condition.

A cylinder $C$ is called $c$-trivial if
there exist an integer $n_0 \ge n(C)$ such that for all $n\ge n_0$
the  morphism $a_{n+1}:\pi_{n+1}(C)\longrightarrow \pi_{n}(C)$
 is a piecewise trivial fibration with fiber $F\cong \res^{c}$.
We denote by $tn(C)$ the least integer $n_0$ verifying such a
condition. Notice that $c$-trivial implies $c$-stable.

\medskip
\begin{proposition}
\label{stable} Let $D$ be a $c$-trivial cylinder of $\MS$ and let
 $C$ be a cylinder.
 Then $D\cap C$ is $c$-trivial and
 $$
\mu_p(C,D):=[\pi_n(C) \cap \pi_n(D)] \; {\mathbb L}^{-c (n+1)} \in
\mathcal M
$$
does not depend on $n$, provided $n \ge tn(D), n(C)$.
\end{proposition}
\begin{proof}
\noindent From the  definition of $c$-trivial cylinder  we get
that
$$
 [\pi_{n+1}(C) \cap \pi_{n+1}(D)] \; {\mathbb L}^{-c (n+2)}
=[\pi_n(C) \cap \pi_n(D)] \;{\mathbb L}^{c}  {\mathbb L}^{-c(n+2)}
=[\pi_n(C) \cap \pi_n(D)] \; {\mathbb L}^{-c (n+1)}.
 $$
\noindent Hence the class  $[\pi_n(C) \cap \pi_n(D)] \; {\mathbb
L}^{-c (n+1)}$ is independent of $n$, provided $n \ge tn(D),
n(C)$.
\end{proof}

\medskip
Let us consider the  cylinder $\MS^{ci}=\pi_n^{-1}(\Sigma_{ci}^n)$,
$n\ge \delta(N,e_0)$, \propref{fibration} $(3)$. By the last
proposition we may define
 on $\mathcal C_{N,p}$ a finitely  additive
measure valued  in ${\mathcal M}$
$$
  \begin{array}{rlc}
  \mu_{p}: {\mathcal C}_{N,p} & \longrightarrow & {\mathcal M} \\
                 C    & \longrightarrow &
               \mu_p(C,\MS^{ci})= [\pi_n(C) \cap \Sigma_{ci}^n] \, \mathbb L ^{-(n+1)(N-1)e_0}
  \end{array}
  $$

\noindent for a big enough $n$.

\medskip
\begin{proposition}
\label{measure} There exists a  finitely  additive
$\widehat{\mathcal M}$-valued measure defined by
$$
\begin{array}{cccc}
    \mu_{p}: &  \mathcal C _{N, p} & \longrightarrow &
    \widehat{\mathcal M}\\
 &  C & \mapsto &\lim_{ i\rightarrow \infty} \mu_p(C,\MS^{ci})
 \end{array}
$$
\end{proposition}
\begin{proof}
Let $C_1,\dots,C_r$ be a disjoint family of cylinders, then we
have to prove
$$
\mu_{p}(C_1\cup \dots \cup C_r)=\Sigma_{i=1}^{r}\mu_p(C_i).
$$
\noindent Notice that we only need to prove the equality for
$r=2$,  this case follows from the group structure of $K_0({\bf
Sch})$.
\end{proof}

In the next proposition  we prove some auxiliary results that we
will use later on. Since $\widehat{ \mathcal M}$ is complete we may
consider the norm $\|\cdot\|: \widehat{ \mathcal M} \longrightarrow
\mathbb R ^+ $, with $\|a\|=2^{-n}$ where $n$ is the order of $a$
with respect the filtration  $\mathcal F$ of $\mathcal M$.

\medskip
\begin{proposition}
\label{arquim} $(1)$ If $D_1$, $D_2$ are cylinders of $\MS$
such that $D_1 \subset D_2$ then
$$
\|\mu_{p}(D_1)\| \le \|\mu_{p}(D_2)\|.
$$

\noindent $(2)$ If $D_1,\dots ,D_r$ are cylinders of $\mathcal H$
then
$$
\|\mu_{p}(D_1 \cup \dots \cup D_r)\| \le Max\{
\|\mu_{p}(D_i)\|, i=1,\dots, r\}.
$$
\end{proposition}
\begin{proof}
$(1)$ We may assume that $D_i=\pi^{-1}_n(T_i)$, $i=1,2$, where $T_1
\subset T_2 \subset \pi_n(\MS)$ are constructible sets, $n\gg 0$. If
we set  $r_i(n)=c(n+1)-dim(T_i\cap \pi_n(\MS))$, $i=1,2$, then it is
easy to prove $r_1(n) \le r_2(n)$ for all $n \gg 0$. Hence we deduce
$$ \|\mu_{p}(D_1)\| = \lim_n 2^{-r_1(n)} \le \lim_n
2^{-r_2(n)}=\|\mu_{p}(D_2)\|. $$

\noindent $(2)$ By induction on $r$ it is enough to prove the result
for $r=2$. Let $d_i=dim(T_i\cap \pi_n(\MS))$, $i=1,2$, then we have
$$ c(n+1)- dim((T_1 \cup T_2)\cap \pi_n(\MS))
 \ge c(n+1)-Max\{d_1, d_2\}.$$ From this it is easy to get the claim.
\end{proof}

\medskip
A subset $C$ of $\MS$ is  measurable with respect $\mu_{p}$
 if and only if for every real number
$\epsilon \in \mathbb R^+$ there exists a sequence of cylindrical
sets $\{C_i\}_{i \ge 0}$ such that $C \vee C_0 \subset
\bigcup_{i\ge 1} C_i$ and $\|\mu_p(C_i)\| \le \epsilon$ for all $i
\ge 1$, where $C \vee C_0$ stands for the disjoint union $C \vee
C_0= C \cup C_0 \setminus C \cap C_0$. We say that $C$ is strongly
measurable if $C_0 \subset C$. See \cite{Bat98}, \cite{DL99c}
Appendix A.

\medskip
\begin{proposition}
\label{DFA} $(1)$ Any cylindrical set is strongly measurable. The
collection of measurable set form a finite algebra of sets
$\mathcal C^*_{N, p}$.

\noindent $(2)$ If $C$ is a measurable set of $\MS$ then $
\mu_{p}(C):= \lim_{\epsilon \rightarrow 0}\mu_p(C_i)\ $ exists
in $\widehat{\mathcal M}$ and is independent of the choice of the
sequence $\{C_i\}_{i \ge 0}$. Then there exists a finite measure
$$
\begin{array}{llll}
 \mu_{p}:  & \mathcal C^*_{N, p} &\longrightarrow &\widehat{\mathcal M} \\
    & C &  \mapsto& \mu_{p}(C)
  \end{array}
$$

\noindent $(3)$ Let $\{C_i\}_{i \ge 0}$ be  a sequence of
measurable sets.

\noindent $(3.1)$ If $\lim_{i \rightarrow \infty}\|\mu_{p}(C_i)\|=0$
then $\cup_{i\ge 0} C_i$ is measurable.

\noindent $(3.2)$ If the sequence of measurable sets are mutually
disjoint and $C=\cup_{i\ge 0} C_i$ is measurable then $\sum_{i \ge
0}\mu_{p}(C_i)$ converges in $\widehat{\mathcal M}$ to
$\mu_{p}(C)$.
\end{proposition}
\begin{proof}
The proof of $(1)$ and  $(3)$ are standard. $(2)$ Let $\epsilon$,
$\epsilon'$ be positive real numbers, and let $\{C_i\}_{i \ge 0}$,
$\{C_i'\}_{i \ge 0}$ sequences of cylindrical sets of $\MS$
verifying the conditions of measurable set. Then we have $ C_0
\vee C_0' \subset \bigcup_{i\ge 1} C_i \cup \bigcup_{i\ge 1} C_i',
$ from \propref{cylinder} $(1)$ there exist integers $r$, $r'$
such that $ C_0 \vee C_0' \subset \bigcup_{i= 1}^{r} C_i \cup
\bigcup_{i=1}^{r'} C_i'. $ \noindent From this and
\propref{arquim} $(2)$ we get $\|\mu_{p}(C_0 \vee C_0')\|
\le Max\{\epsilon, \epsilon'\}$. Notice that $\mu_{p}(C_0)-  \mu_{p}( C_0') = \mu_{p}(C_0 \setminus
C_0')$;
 \propref{arquim} $(1)$ and the last inequality
yield $\|\mu_{p}(C_0)-  \mu_{p}( C_0')\| \le
 Max\{\epsilon,
\epsilon'\}$. From this it is easy to get the claim.
\end{proof}

Let $C$ be a measurable set of $\MS$ and $f: C \longrightarrow
\mathbb Z \cup \{\infty\}$, we say that $f$ is exponentially
integrable if the fibers of $f$ are measurable and the motivic
integral
$$
\int_C \mathbb L ^{-f} d \mu_{p} := \sum_{s \ge
0}\mu_{p}(C \cap f^{-1}(s)) \mathbb L ^{-s}
$$

\noindent converges in $\widehat{\mathcal M}$.

Given a  singularity $X \subset \EAN$ of arbitrary dimension, let us consider the
function
$$
  \begin{array}{cccc}
    \gamma_X: & (\MS)_{\mathrm rat} & \longrightarrow & \mathbb N \cup \{\infty\} \\
              &  C   & \longrightarrow & (C \cdot X)
  \end{array}
$$

\noindent
where $(C \cdot X)$ stands for the "false" intersection
multiplicity: $(C \cdot X)=dim_{\res}(R/I(C)+I(X))$.

\medskip
\begin{proposition}
Let $X \subset \EAN$ be a   singularity of algebraic variety.

\noindent
$(i)$
For all $s \in \mathbb N$ the set
$ \gamma_X^{-1}(s)$ is a cylinder,

\noindent
$(ii)$ $\gamma_X$ is exponentially integrable.
\end{proposition}
\begin{proof}
$(i)$ Notice that if $(C\cdot X)=s$ then $M^s \subset I(C)+I(X)$.
From this fact it is easy to see that  $\gamma_X^{-1}(s)$ is a
cylindrical subset of $\MS$. $(ii)$ We can apply  \propref{DFA},
(3.2), in order to prove that $\gamma_X$ is exponentially
integrable.
\end{proof}

\medskip
Let $X\subset \EAN$ be a singularity of arbitrary dimension. The
motivic volume of $X$ with respect to $p$ is the integral
$$
vol_{p}(X)= \int_{\MS} \mathbb L ^{-\gamma_X} d \mu_{p}
=\sum_{s \ge 0} \mu_{p}(\gamma_X^{-1}(s))
\mathbb L ^{-s}
$$

Given an integer $e_0$ we denote by $\mathcal H(e_0)$ the finite set
of admissible Hilbert polynomials, see \propref{compositio}.

\medskip
\begin{theorem}[definition of motivic volume]
Let $X$ be  a singularity,  then the series
$$ \sum_{e_0 \ge 1}
    \Bigl( \sum_{p\in \mathcal H(e_0)} vol_{p}(X) \Bigr) {\mathbb L}^{- e_0 }
$$

\noindent converges in $\widehat{\mathcal M}$ to the motivic volume $vol(X)$ of $X$.
\end{theorem}
\begin{proof}
Let us consider the motivic volume of $X$ with respect $p$
$$
\begin{array}{ll}
vol_{p}(X)& = \sum_{s \ge 0}\mu_{p}(  \gamma^{-1}(s)) \mathbb L ^{-s}\\
              &  = \sum_{s \ge 0}[\pi_n( \gamma^{-1}(s)) \cap
              \Sigma^{n}_{ci}] {\mathbb L}^{-((n+1)(N-1) e_0 +s)}
\end{array}
$$

\noindent
with $n \ge s, \delta(N,e_0)$.
From \propref{fibration} we get that
$$
[\pi_n( \gamma^{-1}(s)) \cap \Sigma^{n}_{ci}]
 {\mathbb L}^{-((n+1)(N-1) e_0 +s)} \in F^{l} \mathcal M
$$

\noindent
for a non-negative integer $l$.
From this we get

$$
\Bigl( \sum_{p\in \mathcal H(e_0))} vol_{p}(X) \Bigr) {\mathbb L}^{- e_0 }
\in F^{e_0} \mathcal M
$$

\noindent
and we are done.
\end{proof}

\medskip
Let $\mathcal P$ be a property defined in the set of curve
singularities with Hilbert polynomial $p$. Let $c(\mathcal P)$ be
the set of rational points of $\MS$ corresponding to curve singularities
verifying the property $\mathcal P$. We say that $\mathcal P$ is
finitely determined if there exists an integer $n_0=n_0(\mathcal
P)$, that we may assume $n_0\ge \delta(N,e_0)+1$, \propref{fibration},
such that for all curve singularities $C^1$, $C^2$ with
Hilbert polynomial $p$ and $C^1_n=C^2_n$, $n \ge n_0$, then $C^1$
verifies  $\mathcal P$ if and only if $C^2$  verifies $\mathcal
P$. Notice that any $\mathcal P$  analytically invariant property
defined in the set of reduced curve singularities is finitely
determined, \cite{Eli86b}.

We say that a finitely determined property $\mathcal P$ is
constructible, cfd-property for short,
if the set of truncations $[C_n] \in \Xi_n$, $n\ge
n_0(\mathcal P)$, such that $C$ verifies $\mathcal P$ is a
constructible set ${\mathbf c}_n(\mathcal P)$ of $\Xi_n$.
Trivially  a cfd-property $\mathcal P$ is determined by the
cylindrical set $c(\mathcal P)=\pi^{-1}_n{\mathbf c}_n(\mathcal P)$.
On the other hand every  cylindrical set defines a cfd-property.

Given a  property $\mathcal P$ of curve singularities with Hilbert
polynomial $p$  we define its motivic  Poincare  series by,
$n_0=n_0(\mathcal P)$,
$$
MPS_{\mathcal P}= \sum_{n\ge n_0} [\pi_n(c(\mathcal P))] T^n \in {\mathcal
M}[[T]].
$$

\medskip
We denote by  ${\mathcal M}[T]_{loc}$  the ring of rational power series, i.e.
the sub-ring of ${\mathcal M}[[T]]$
generated by ${\mathcal M}[T]$ and  $(1- {\mathbb L}^{a}  T^b)^{-1}$, $a \in \mathbb N$,
$b \in \mathbb N \setminus \{0\}$.

\medskip
\begin{proposition}
\label{motivicpoincare}
Let $\mathcal P$ be a  cfd-property then it holds
$$
MPS_{\mathcal P} \in {\mathcal M}[T]_{loc}.
$$
\end{proposition}
\begin{proof}
If we set $n_0=n_0(\mathcal P)$ then we have
$$
\begin{array}{ll}
 MPS_{\mathcal P}
    &=  \sum_{n\ge n_0} [\pi_n(c(\mathcal P))\cap \Sigma^n_{ci}] T^n \\
    &=  \sum_{n\ge n_0} [\pi_{n_0}(c(\mathcal P))\cap \Sigma^{n_0}_{ci}]
    {\mathbb L}^{(N-1) e_0 n}     T^{n} \\
    &=  [\pi_{n_0}(c(\mathcal P))\cap \Sigma^{n_0}_{ci}]
    \frac{ {\mathbb L}^{ (N-1) e_0 n_0} T^{ n_0}}{1- {\mathbb L}^{(N-1) e_0 } T}
\end{array}
$$
\noindent
and we get the claim.
\end{proof}

Notice that the condition
 "{\em belongs to $\MS$}" is a cfd-property that we
will denote by $\MS$. A simple computation shows:

$$
 MPS_{\MS} =
    [ \Sigma^{n_0}_{ci}]
    \frac{ {\mathbb L}^{ (N-1) e_0 n_0} T^{ n_0}}{1- {\mathbb L}^{(N-1) e_0 } T}
$$
\noindent
where $n_0=\delta(N,e_0)+1$.

\bigskip
\section{Local  properties of the moduli space.}
\label{localglobal}

\medskip
The  purpose of this section is to study the local ring ${\mathcal
O}_{\MS,x}$ where $x$ is a closed point of $\MS$. In particular we
will compute the tangent space $T_{x}= Hom_{{\bf k}}({\bf
n}_{x}/{\bf n}_{x}^{2},{\bf k})$ of $\MS$  at $x$, where ${\bf
n}_{x}$ is the maximal ideal of ${\mathcal O}_{\MS,x}$.

\vskip 2mm We denote by $\bf{ Aff'}$ the subcategory of $\bf{
Aff}$ of the ${\bf k}-$schemes $Spec(A)$ where $A$ is an Artinian
local ${\bf  k}-$algebra. Let $\underline{H}_{N,p(T)}^{x}$ be the
contravariant functor between $\bf{ Aff'}$ and $\bf{ Set}$, such
that for any object $S$ of $\bf{ Aff'}$ we have

$$ \underline{H}_{N,p(T)}^{x}(S)= \left\{ \hskip 1mm
\parbox{5cm}{
$Z \in \underline{\bf{H}}_{N,p(T)}(S)$  such that

\medskip
\centerline{$Z_{s}=C_{x}$}

\medskip
for the closed point $s$ of $S$. } \hskip 1mm \right\} $$

\medskip
We will denote by $D$ the ${\bf  k}-$scheme $Spec({\bf
k}[\varepsilon ])$, let $p_{D}$ be the closed point of $D$. Notice
that the elements of  $\underline{H}_{N,p(T)}^{x}(D)$ are first
embedded deformations of $C_{x}$, but not all, see
\exref{flat-nfam}. It is well known that there exists a bijection
between $T_{x}$ and the set of morphism $f:D \to \MS$ such that
$f(p_{D})=x$. By \thmref{monster} we obtain that there exist a
bijection between $T_{x}$  and $\underline{H}_{N,p(T)}^{x}(D)$.
From \thmref{monster} it is easy to prove

\medskip
\begin{proposition}
The functor $\underline{H}_{N,p(T)}^{x}$ is pro-represented by
$Spec({\mathcal O}_{\MS,x})$.
\end{proposition}

\medskip
In order to compute $T_{x}$ we need to characterize the set
$\underline{H}_{N,p(T)}^{x}(D)$. Let $I$ be an ideal of $R$, we
consider  Hironaka's invariant $v^{*}(I)$ associated to $I$, see
\cite{Hir64a} Chapter III, definition 1. Let $f_{1},...,f_{s}$ be a
standard basis of $I$ such that $order(f_{i})=v^{i}(I)$. Recall that
from \cite{Eli86b}, Proposition 2, we have that $v^{i}(I(C)) \le
e_0$, $i=1,...,s$, for all curve singularity $C$  of multiplicity
$e_0$.

\medskip
\begin{proposition}
\label{chardefs} Let $J=(f_{1}+\varepsilon
g_{1},...,f_{s}+\varepsilon g_{s})$ be an ideal of ${\bf
k}[\varepsilon ][[\underline X]]$ defining a first order deformation
$\varphi :Z \to D$ of a curve singularity $C$ of multiplicity $e_0$
 defined by the ideal $I=(f_1,\dots, f_s)$. Then the following
conditions are equivalent:
\begin{enumerate}
\item[(1)]
 $Z$ is a  family of curve singularities with Hilbert
polynomial $p(T)$,
\item[(2)]
$\varphi :Z_{e+1} \to D$ is flat,
\item[(3)]
 for all $i=1,...,s$ it holds that
$g_{i} \in (I+M^{e_0+1}:I+M^{e_0+1-v^{i}(I)})$,
\end{enumerate}
\end{proposition}
\begin{proof}
We put $\nu^i=\nu^i(I)$.
We will use the syzygy flatness criterion, see for instance
\cite{Art76}, Corollary page 11. From this result, we deduce that
$E_{n}={\bf k}[\varepsilon ][[\underline X]]/(J+M^{n})$ is a flat
${\bf k}[\varepsilon ]-$module if and only if for all
$a_{1},...,a_{s} \in R$ such that $\sum_{i=1}^{s}a_{i}f_{i} \in
M^{n}$ there exist $A_{1},...,A_{s} \in R$ such that
$$\sum_{i=1}^{s}(a_{i}+\varepsilon A_{i})(f_{i}+\varepsilon g_{i})
\in M^{n}{\bf k}[\varepsilon ][[\underline  X]].$$

\noindent By definition of family of curve singularities we get that $(1)$ implies $(2)$.

\noindent $(2)$ implies $(3)$. Let us assume that $\varphi
:Z_{e_0+1} \to D$ is flat. Let $a$ be an element of
$M^{e_0+1-v^{i}}$, we need to show that for all $i=1,...,s$  we have
$g_{i}a \in I+ M^{e_0+1}.$ From the flatness of $E_{e_0+1}$ and
$f_{i}a \in M^{e_0+1}$, we deduce that there exists $A_{1},...,A_{s}
\in R$ such that $g_{i}a+ \sum_{i=1}^{s}A_{i}f_{i} \in M^{e_0+1},$
so $g_{i}a \in M^{e_0+1}+I$.

\noindent
$(3)$ implies $(1)$.
 Suppose that for all $i=1,...,s$ it holds $$g_{i} \in
(I+M^{e_0+1}:I+M^{e_0+1-v^{i}}),$$

\noindent so for all $n \ge e_0+1$ we get $g_{i} \in
(I+M^{n}:I+M^{n-v^{i}}).$ Let  $a_{1},...,a_{s} \in R$ be elements
with $\sum_{i=1}^{s}a_{i}f_{i} \in M^{n}$. Since $f_{1},...,f_{s}$
is an standard basis of $I$, by \cite{RV80}, Corollary 1.8, there
exists $C_{i} \in M^{n-v^{i}}$, $i=1,...,s$, such that $
\sum_{i=1}^{s}a_{i}f_{i}=\sum_{i=1}^{s}C_{i}f_{i}, $
 so $a_{1}-C_{1},...,a_{s}-C_{s}$ is a syzygy of
$f_{1},...,f_{s}$.
 From the flatness of $Z$ over $D$, we deduce
that there exist $A_{1},...,A_{s}$ such that
$
\sum_{i=1}^{s}(a_{i}-C_{i}+\varepsilon A_{i})(f_{i}+\varepsilon
g_{i})=0. $
From the assumption $(2)$ there exist $B_{1},...,B_{s}
\in R$ such that
$$ \sum_{i=1}^{s} C_{i}g_{i}-\sum_{i=1}^{s} B_{i}f_{i} \in M^{n}.
$$
\noindent Since $\sum_{i=1}^{s}a_{i}f_{i} \in M^{n} $, from the
two last equalities it is easy to see that

$$ \sum_{i=1}^{s}(a_{i}+\varepsilon
(A_{i}-B_{i}))(f_{i}+\varepsilon g_{i}) \in M^{n}. $$

\noindent We have proved that $E_{n}$ is a $D-$module flat for all
$n \ge e_0+1$, so $Z$ is a family.
 \end{proof}

\medskip
\begin{corollary}
For all closed point $x$ of $\MS$, the morphism $$ d(\pi _{n}):
T_{\MS, x} \longrightarrow T_{\Xi _{n}, \pi _{n}(x)} $$ \noindent
is surjective, $n \ge e_0+1$.
\end{corollary}

\medskip
\propref{chardefs} enable us  to give an example of a  first order
deformation that is not a  family.

\medskip
\begin{example}
\label{flat-nfam} Let  us consider the plane curve singularity
defined by the equation $X_{1}^{3}=0$. In this case we have
$v_{1}=3$ and $e_0=3$. Let us consider the first order deformation
$Z$  defined by $ X_{1}^{3} + \varepsilon X_{1}, $ i.e.
$g_{1}=X_{1}$. From the last result we get that $Z_{4} \to D$ is
not flat.
\end{example}

\medskip
It is well known that the first order embedded deformations of $C_{x}$
are classified by the normal module
$$N_{x}=Hom_{R/I_{x}}(I_{x}/I_{x}^{2},R/I_{x}).$$

\noindent
 If we denote by
${\bf embdef}(C_{x})$
the first order embedded deformations of $C_{x}$, we will define a
bijective map $\tau : N_{x} \longrightarrow {\bf embdef}(C_{x}).$
Let $g:I_{x}/I_{x}^{2} \to R/I_{x}$ be a morphism of $R/I_{x}-$modules.
Then a lifting of $g$ is a $R/I_{x}-$module morphism
$(\overline{g}_{.})=
(\overline{g}_{1},...,\overline{g}_{s}):(R/I_{x})^{s} \to
R/I_{x}$ such that the following diagram is commutative
$$
\xymatrix{
(R/I_{x})^{s}  \ar[d]_{(f_{.})}   \ar[dr]^{\relax (\overline{g}.)} &         \\
I_{x}/I_{x}^{2}   \ar[r]_{g}                       & R/I_{x} \\
}
$$
\noindent $\tau (g)$  is the first order deformation defined by
the ideal $J=(f_{1}+\varepsilon g_{1},...,f_{s}+\varepsilon
g_{s})$
 with $(g_{1},..,g_{s})$  a lifting of $g$.

\medskip
We denote by $N_{x}'$ the set of $g \in N_{x}$ for which there
exist a lifting $(g_{1},..,g_{s})$ such that
 for all $i=1,...,s$ it holds that
$\overline{g}_{i} \in ({\bf m}^{e_0+1}_{x}:{\bf
m}_{x}^{e_0+1-v^{i}}).$
 From last Proposition we deduce

\medskip
\begin{proposition}
$\tau $ defines a bijection between $T_{x}$ and $N_{x}'.$
\end{proposition}

\medskip
\noindent In a very few cases we can compute $v^{*}(I(C))$; for
example if $C$ is a curve singularity with  maximal Hilbert function
then $v^{*}(I(C)) = \{e_0, \dots, e_0+1\}$, see \cite{Ore83}. Notice
in this case the ring $Gr_{{\bf m}}({\mathcal O}_{C})$ is
Cohen-Macaulay, $\bf m$ is the maximal ideal of  ${\mathcal O}_{C}$.
We can prove a more general result without any restriction on
$v^{*}$.

\medskip
\begin{proposition}
Let $C$ be a curve singularity of ${({\bf k}^N,0)}$ of multiplicity
$e_0$ and  embedding dimension $b$. It holds:

\noindent $(1)$ If the associated graded ring $Gr_{{\bf
m}}({\mathcal O}_{C})$ is Cohen-Macaulay then $({\bf m}^{e_0+1}:{\bf
m}^{e_0+1-v^{i}})={\bf m}^{v^{i}}$ for $i=1,...,s$.

\noindent
$(2)$
 If $C$ has a maximal Hilbert function and $e_0 =
\binom{b-1+r}{r}$, for some integer $r$, then $I(C)$ has an standard
basis of $s=\binom{b-1+r}{r+1}$ forms of degree $r+1$ and the ring
$Gr_{{\bf m}}({\mathcal O}_C)$ is Cohen-Macaulay. In particular
$({\bf m}^{e_0+1}:{\bf m}^{e_0-r})={\bf m}^{r+1}$.

\end{proposition}
\begin{proof}
$(1)$ We put $v^i= v^i(I)$. Notice that we always have ${\bf
m}^{v^{i}} \subset ({\bf m}^{e_0+1}:{\bf m}^{e_0+1-v^{i}}).$ Let $h$
be an element of ${\mathcal O}_{C}$ such that $h \mbox{ }{\bf
m}^{e_0+1-v^{i}} \subset {\bf m}^{e_0+1}$, $i=1,..,s$.

\noindent Suppose that $h \not\in  {\bf m}^{v^{i}}$; let $t$ be
the least integer such that $h \in {\bf m}^{t} \setminus {\bf
m}^{t+1}$.
Let $x$ be a degree one superficial element  of
${\mathcal O}_{C}$.
Then $x$ defines a non-zero divisor of
 $gr_{{\bf m}}({\mathcal O}_{C})$.
This implies that $ {\bf m}^{t}/{\bf m}^{t+1}
\stackrel{x^{e_0-t}}{\longrightarrow} {\bf m}^{e_0}/{\bf m}^{e_0+1}$
is a monomorphism. Notice that $e_0-t \ge e_0+1-v^{i}$, so
 $h \overline{x}^{e_0-t} \in {\bf m}^{e_0+1} $.
Since the coset of $h$ in ${\bf m}^{t}/{\bf m}^{t+1}$ is non-zero
we get a contradiction. Hence $h \in {\bf m}^{v^{i}} $, and we get
the result.

\noindent
$(2)$
We only need to prove that $s=\binom{b-1+r}{r+1}$ and
$v^{i}(I(C))=r+1$ for $i=1,...,s$. This follows from \cite{Ore83}.
\end{proof}

\medskip
Notice that very often if $p$ is a rigid Hilbert polynomial then
the associated graded ring is Cohen-Macaulay, see for instance
\cite{EV91}, \cite{Eli99}.

\medskip
\begin{corollary}
Let $x$ be a closed point of ${\bf H}_{2,p(T)}$. Then there exists
a natural bijection  between $T_{x}$ and ${\bf m}_{x}^{e_0}$.
\end{corollary}

\medskip
\medskip
We will end this paper studying the local structure  of the closed
points of $\MS$.

\medskip
\begin{definition}
Let $C$ be a curve singularity with Hilbert polynomial $p(T)$. We
say that $C$, or the corresponding closed point $x$ of $\MS$, is
non-obstructed if the functor $ \underline{H}_{N,p(T)}^{x}$ is
smooth, i.e. for all epimorphism of local finitely generated
Artinian ${\bf k}$-algebras $h: A \longrightarrow A^{'}$ the set map
$$
\xymatrix@C=2cm{
\underline{H}_{N,p(T)}^{x}(Spec(A))
\ar[r]^{\underline{H}_{N,p(T)}^{x}(h)} &
\underline{H}_{N,p(T)}^{x}(Spec(A^{'}))\\
}
$$
\noindent is
surjective.
\end{definition}

\medskip
Let $\kmod$ be the category of ${\bf k}$-modules. Given a ${\bf
k}$-module $W$ we will denote by $W^{*}=Hom_{{\bf k}}(W,{\bf k})$
its dual space. It is well known that there exists a natural
monomorphism $ W \longrightarrow W^{*}, $ that it is isomorphism
if and only if $W$ is a finite dimensional ${\bf k}$-module.

Following  Laudal, \cite{Lau79}, pag. 102, we will consider each
object of $\kmod$ endowed with a topology that will induce its
reflexivity. Let $W$ be an object of $\kmod$, pick a ${\bf k}$-base
${\mathcal V}=\{v_{i}\}_{i \in I}$. We put on $W$ the topology for
which a basis of neighborhoods of the neutral elements consists of
the subspaces of $W$ containing all but a finite number of the
elements of ${\mathcal V}$. We will denote by $\ktopmod$ the
corresponding category of topological ${\bf k}$-modules. Given an
object of $\ktopmod$. We will denote by $W^{\circ}$ the topological
dual of $W$. It is easy to see that the dual basis of ${\mathcal V}$
defines a topology on $W^{\circ}$ such that $ W \cong W^{\circ
\circ}. $

We know that for all closed point $x$ of $\MS$ it holds $$ {\mathcal
O}_{\MS , x} \cong \ilim \pi ^{*}_{n}({\mathcal O}_{\Xi _{n},
x_{n}}), $$ \noindent where $x_{n}=\pi _{n}(x)$, \cite{EGA-IV},
8.2.12.1. In particular we have $ {\bf n}_{x} \cong \ilim \pi
^{*}_{n}({\bf n}_{x_{n}}), $ where ${\bf m}_{x_{n}}$ is the maximal
ideal of ${\mathcal O}_{\Xi _{n}, x_{n} }$. For all $n \ge e_0+1$ we
pick a ${\bf k}-$basis ${\mathcal V}_{n}=\{e_{j}\}_{j \in {\mathcal
J}_{n}}$ of ${\bf m}_{x_{n}}$, such that ${\mathcal J}_{n} \subset
{\mathcal J}_{n+1}$, and ${\mathcal V}_{n} \subset {\mathcal
V}_{n+1}$. Notice that $ {\mathcal V} = \ilim {\mathcal V}_{n}, $ is
a ${\bf k}-$basis of ${\bf n}_{x}$, and ${\mathcal V} \cup \{1\}$ is
a ${\bf k}$-basis of ${\mathcal O}_{\MS , x}$. We write ${\mathcal
J} = \ilim {\mathcal J}_{n}.$
 From now on we
will consider ${\mathcal O}_{\MS , x}$ endowed with the topology for
which a basis of neighborhoods of the neutral element are the ideals
containing all but a finite number of elements of ${\mathcal V}$.
This topology will permit us to characterize the non-obstructiveness
of the closed points of $\MS$, see \thmref{non-obs}. We will denote
by ${\mathcal O}_{\MS , x}^{+}$ its completion with respect the
topology defined above. This topology induces a topology on ${\bf
n}_{x}/{\bf n}_{x}^{2}$, making this ${\bf k}$-module an object of
$\ktopmod$. Moreover, this topology induces also a topology on the
tangent space $T_{x}$, in which the neighborhoods of the neutral
element of $T_{x}$ are the linear maps $ \omega : {\bf n}_{x}/{\bf
n}_{x}^{2} \longrightarrow {\bf k} $ whose kernel contains all but a
finite number of elements of ${\mathcal V}$.

\medskip
\noindent Let us consider the ${\bf k}$-algebra morphism $\varphi :
{\bf k}[T_{i},i \in {\mathcal V}] \longrightarrow {\mathcal
O}_{\MS,x} $ such that $\varphi (T_{i})=e_{i}$, $i \in {\mathcal
J}$. Notice that $ \varphi_{2} : {\bf k}[T_{i},i \in {\mathcal V}]
/(T_{i})^{2} \longrightarrow {\mathcal O}_{\MS,x}/{\bf n}_{x}^{2} $
is an isomorphism of ${\bf k}$-modules. We will consider in  ${\bf
k}[T_{i},i \in {\mathcal V}]$ the topology for which a basis of the
neutral element are the ideals $I$ contained in $(T_{i})$ such that
all but a finite number of $T_{i}$ belongs to $I$. We will denote by
${\bf k}[T_{i},i \in {\mathcal V}]^{+}$ the completion of ${\bf
k}[T_{i},i \in {\mathcal V}]$ with respect this topology. Notice
that $\varphi $ is a continuous map with respect to the topologies
defined above.

\medskip
\begin{theorem}
\label{non-obs} A closed point $x$   of $\MS$ is  a non-obstructed
point if and only if
$$\varphi ^{+} : {\bf k}[T_{i}; i\in I]^{+}
\longrightarrow {\mathcal O}_{\MS, x}^{+}$$
\noindent is an
isomorphism of ${\bf k}$-algebras.
\end{theorem}
\begin{proof}
We will write $S={\bf k}[T_{i},i \in {\mathcal V}]$, and ${\mathcal
O}_{x}={\mathcal O}_{\MS,x}$. We will denote by $I_{n}$ the ideal of
$S$ generated for  $T_{i}$, $i \in {\mathcal J} \setminus {\mathcal
J}_{n}$; and we will denote by $W_{n}$ the corresponding ideal of
${\mathcal O}_{x}$. Hence $\varphi $ induces an epimorphism of ${\bf
k}$-vector spaces $ \varphi_{n} : S/I_{n} \longrightarrow {\mathcal
O}_{x}/W_{n}. $ Since the set of $\{I_{n}\}_{n}$, resp. $W_{n}$, are
cofinal in  the basis of neighborhoods of $S$, resp. ${\mathcal
O}_{x}$, we get that $\varphi ^{+}$ is an isomorphism of ${\bf
k}$-algebras if and only if $\varphi _{n}$ is an isomorphism for a
$n$ big enough.

Let us assume that $x$ is non-obstructed and that $\varphi _{n}$
is not an isomorphism; let $F$ be a non-zero element of $S/I_{n}$
belonging to the Kernel of $\varphi _{n}$. Let  $s$ be an integer
such that the coset of $F$ in $A^{'}=S/I_{n}+(T_{i})^{s}$
is non-zero. Let $h$ be morphism induced by $\varphi
_{n}$
$$ h : A^{'}=S/I_{n}+(T_{i})^{s} \longrightarrow A={\mathcal
O}_{x}/W_{n}+{\bf n}_{x}^{s}. $$ Let us consider the projection $\pi
: {\mathcal O}_{\MS,x} \longrightarrow A$. Since $x$ is a
non-obstructed point  we get that there exists a morphism $\sigma
:{\mathcal O}_{\MS,x} \longrightarrow A^{'}$ such that the following
diagram is commutative

$$
\xymatrix{
{\mathcal O}_{\MS, x}    \ar[d]_{\relax \sigma}   \ar[dr]^{\relax \pi} &         \\
A^{'}  \ar[r]_{h}                       & A
}
$$

\noindent Since $\varphi _{2}$ is an isomorphism we get that the
above diagram induces

$$
\xymatrix{
A  \ar[d]_{\relax \overline{\sigma}}   \ar[dr]^{\relax \overline{\pi}=Id} &         \\
A^{'}  \ar[r]_{h}                       & A
}
$$

\noindent so $h$ is an isomorphism and we get a contradiction.

If $\varphi ^{+}$ is an isomorphism then it is easy to see that
$x$ is non-obstructed.
\end{proof}

\medskip
\begin{proposition}
(i)  Every closed point $x$ of ${\bf H}_{2,p(T)}$, with $p(T)= e_0T-e_0(e_0-1)/2$,
 is non-obstructed.
In particular ${\bf H}_{2,p(T)}$ is reduced.

\noindent (ii) Given an integer $a \ge 2$ we put
$p_{a}=\binom{a}{2 } T - 2 \binom{a}{3}$. Then every closed point $x$ of
${\bf H}_{3,p_{a}}$ is non-obstructed, and  ${\bf
H}_{3,p_{a}}$ is reduced.
\end{proposition}
\begin{proof}
(i) Let $h: A \longrightarrow A^{'}$ be a epimorphism of local
finitely generated local ${\bf k}$-algebras; we need to prove that
$\underline{H}_{2,p(T)}^{x}(h)$ is surjective. We write
$S=Spec(A)$ and $S^{'}=Spec(A^{'})$.

Given a family $Z^{'} \in \underline{H}_{2,p(T)}^{x}(S)$ there
exists $F \in A^{'}[[X_{1},X_{2}]]$ such that
$Z^{'}=Spec(A^{'}[[X_{1},X_{2}]]/(F)).$
Let $G \in A[[X_{1},X_{2}]]$ be a power series of order
$e$ and such that $h(G)=F$. Then it is easy to prove that the
family $Z=Spec(A[[X_{x},X_{2}]]/(G))$ verifies
$\underline{H}_{2,p(T)}^{x}(h)(Z)=Z^{'}$.

\noindent (ii) Let $h: A \longrightarrow A^{'}$ be a epimorphism
of local finitely generated local ${\bf k}$-algebras; we write
$S=Spec(A)$ and $S^{'}=Spec(A^{'})$. Let  $Z^{'}$ be a family of
curve singularities of ${({\bf k}^3,0)}$ over $S^{'}$ with Hilbert
polynomial $$p_{a}(T)= \binom{a}{2 } T - 2 \binom{a}{3}.$$

\noindent Let $x$ be the closed point of ${\bf H}_{3,p_{a}(T)}$
defined by the closed fiber of $Z^{'}$. From Corollary 4.5, we get
that $v(I_{x})=a$ and $I_{x} \subset R={\bf
k}[[X_{1},X_{2},X_{3}]]$ admits a minimal free resolution $$ 0
\longrightarrow R^{a-1} \stackrel{M}{\longrightarrow} R^{a}
\longrightarrow I_{x} \longrightarrow 0 $$

\noindent with $M$ a matrix with entries belonging to the maximal
ideal of $R$ of order $1$. Hence there exists an $a \times (a-1)$
matrix $\overline{M}$ with coefficients in the maximal ideals of
$A[[X_{1},X_{2},X_{3}]]$, such that its maximal minors generates
and ideal, say $I^{'}$, with
$Z^{'}=Spec(A^{'}[[X_{1},X_{2},X_{3}]]/I^{'})$
and $\overline{M}$ is mapped to $M$ by the natural
morphism of ${\bf k}-$algebras
$A^{'}[[X_{1},X_{2},X_{3}]]\longrightarrow {\bf
k}[[X_{1},X_{2},X_{3}]].$
 Let $N$ be an $a\times (a-1)$ matrix
such that its entries belong to the maximal ideal of $A$ and are
mapped to the entries of $\overline{M}$ by $h$. Let $I$ be the
ideal of $A[[X_{1},X_{2},X_{3}]]$ generated by the maximal minors
of $M$; we put $Z=Spec(A[[X_{1},X_{2},X_{3}]]/I)$. It is easy to
see that $Z \in \underline{H}_{3,p_{a}(T)}^{x}(S)$, where
$S=Spec(A)$, and $Z^{'}=\underline{H}_{3,p_{a}(T)}^{x}(h)(Z).$
\end{proof}


\end{document}